\documentclass[a4paper,12pt]{article}
\usepackage{amsmath,amssymb,amsthm,cite,verbatim,xcolor,tikz-cd}
\usepackage[bookmarks=true]{hyperref}
\usepackage{graphicx}
\title{Bielliptic quotient modular curves of $X_0(N)$}

\author{ \ Francesc Bars\footnote{First author is supported by PID2020-116542GB-I00.}, Mohamed Kamel
and Andreas Schweizer}
\newtheorem{prop}{Proposition}[section]
\newtheorem{lema}[prop]{Lemma}
\newtheorem{teo}[prop]{Theorem}
\newtheorem{cor}[prop]{Corollary}
\newtheorem{rem}[prop]{Remark}

\theoremstyle{definition}

\theoremstyle{remark}

\numberwithin{equation}{section}

\newcommand{\Q}{\mathbb{Q}}
\newcommand{\QQ}{\mathbb{Q}}

\newcommand{\Z}{\mathbb{Z}}

\newcommand{\F}{\mathbb{F}}
\newcommand{\FF}{\mathbb{F}}

\newcommand{\C}{\mathbb{C}}

\newcommand{\Gal}{\mathrm{Gal}}
\newcommand{\GL}{\operatorname{GL}}

\newcommand{\Aut}{\operatorname{Aut}}

\newcommand{\Jac}{\operatorname{Jac}}

\newcommand{\New}{\operatorname{New}}

\newcommand{\cA}{{\mathcal A}}
\newcommand{\cB}{{\mathcal B}}

\newcommand{\cS}{{\mathcal S}}

\newcommand{\cL}{{\mathcal L}}

\usepackage{xcolor}
\newcommand{\colorredbold}[1]{\bold{{\color{red}{#1}}}}

\date{}
\begin{document}
\maketitle

\begin{abstract}
\noindent Let $N\geq 1$ be a non-square free integer and let $W_N$
be a non-trivial subgroup of the group of the  Atkin-Lehner
involutions of $X_0(N)$ such that the modular curve $X_0(N)/W_N$ has
genus at least two. We determine all pairs $(N,W_N)$ such that
$X_0(N)/W_N$ is a bielliptic curve and the pairs $(N,W_N)$ such that
$X_0(N)/W_N$ has an infinite number of quadratic points over
$\mathbb{Q}$.

\end{abstract}

\renewcommand{\thefootnote}{\fnsymbol{footnote}}
\renewcommand{\thefootnote}{\arabic{footnote}}

\renewcommand{\thefootnote}{\fnsymbol{footnote}}
\renewcommand{\thefootnote}{\arabic{footnote}}

\section{Introduction}

Let $X$ be a smooth projective curve of genus $g_X\geq 2$ defined
over a number field $K$. We also assume that $X$ has at least one
$K$-rational point. (All curves we investigate in this paper have
this property.)
\par
The curve $X$ is called  {\bf bielliptic} if it has an involution
$v$ such that $X/v$ has genus $1$. More precisely, $X$ is called
bielliptic over $L$ for an extension $L$ of $K$ if $v$ is defined
over $L$. In this case $X/v$ inherits an $L$-rational point from $X$
and hence is an elliptic curve over $L$.
\par
Although for every finite extension $L$ of $K$ the set $X(L)$ of
$L$-rational points of $X$ is finite by a famous theorem by
Faltings, the set
$$\Gamma_2(X,K):=\bigcup_{[L:K]\leq 2}X(L)$$
of quadratic points over $K$ can be infinite. This happens for
example if $X$ is bielliptic over $K$ and the elliptic curve $X/v$
has positive rank over $K$, because then over each of the infinitely
many $K$-rational points of $X/v$ there is at least one quadratic
point on $X$. By a similar argument $\Gamma_2(X,K)$ is infinite if
$X$ is hyperelliptic, because then the hyperelliptic involution $u$
is defined over $K$ and $X/u\cong\mathbb{P}^1$.
\par
Less obvious is that the converse also holds, i.e. $\Gamma_2(X,K)$
is infinite if and only if $X$ is hyperelliptic or $X$ is bielliptic
over $K$ with elliptic quotient curve of positive $K$-rank. For
details see \cite{BaMom}, also for the more complicated situation
where $X(K)=\emptyset$. Note that the similar statement in
\cite{SiHa} is slightly weaker in that (at least for $g_X\leq 5$) it
is not stated over which field the bielliptic involution is defined.
\par
Therefore an investigation which curves of a certain type or in a
certain family have infinitely many quadratic points often starts
with determining the bielliptic ones, the determination of the
hyperelliptic ones usually having been done decades earlier.
\par
This paper is the final installment in a series that investigates
the two problems, biellipticity and infinitely many quadratic points
over $\QQ$, for curves $X_0(N)/W_N$ where $X_0(N)$ is a modular
curve of Hecke type and $W_N$ is a subgroup of the group $B(N)$ of
Atkin-Lehner involutions of $X_0(N)$. Note that the AL-involutions
are defined over $\QQ$, and so $X_0(N)/W_N$ inherits a
$\QQ$-rational cusp from $X_0(N)$.
\par
For the trivial subgroup $W_N =\{1\}$ it was already \cite{Ba1}
which determined all bielliptic curves $X_0(N)$ and those with
infinite $\Gamma_2(X_0(N),\QQ)$. Later \cite{Jeon} solved the same
problem for the curves $X_0^+(N)=X_0(N)/w_N$ where $w_N$ is the
Fricke involution. Then \cite{BaGon} and \cite{BaGon2} treated the
curves $X_0^*(N)=X_0(N)/B(N)$ for square-free resp. non-square-free
$N$.
\par
After the case of proper subgroups $W_N$ was settled in
\cite{BaGonKa} for square-free $N$, we now deal with proper
subgroups $W_N$ when $N$ is not square-free.
\par
In Section 3, specializing a result by Harris and Silverman we note
that $X_0(N)/W_N$ only has a chance of being bielliptic when
$X_0^*(N)$ is bielliptic, hyperelliptic or of genus at most $1$. So
with \cite{BaGon2}, \cite{Ha97} and \cite{GL} we can restrict the
candidates to a finite, explicit set of $N$, albeit with usually $7$
curves for each such $N$.
\par
Gonz\'alez and the first author in \cite{BaGon2} provided a
computational method by use of a theorem of Petri to decide whether
a concrete curve is bielliptic or not, for modular curves after
computing the Jacobian and the Galois conjugation basis. But the
computations are tedious and not transparent to the reader.
Therefore we first use several old and new criteria (see Section
4.1) to quickly discard as many candidates as possible.
\par
A bielliptic curve of genus smaller than $6$ might have several
bielliptic involutions, and, worse, it can happen that none of these
is defined over $\QQ$. But if $N$ is square-free, then all
automorphisms of $X_0(N)/W_N$ are known to be defined over $\QQ$,
which is quite helpful. For non-square-free $N$ this does not
necessarily hold.
\par
On the other hand, most non-square-free candidates $N$ are divisible
by $4$ or $9$, and in these cases there exist additional involutions
of $X_0(N)$ coming from the normalizer of $\Gamma_0(N)$ in
$GL_2(\mathbb{R})$. This gives us more candidates for bielliptic
involutions. Indeed, for practically all bielliptic $X_0(N)/W_N$
with $N$ not square-free we can write down explicit bielliptic
involutions, as matrices if one wants to.
\par

Our first main result is
\begin{teo}\label{mainpq} Let $N>1$ be a non square-free integer not a power of a prime.
Assume that the genus of the modular curve $X_0(N)/W_N$ is at least
$2$ for a  non-trivial subgroup $W_N$ of $B(N)$ different from  $
\langle w_N\rangle$. The curve  $X_0(N)/W_N$, denoted by the pair
$(N,W_N)$ is a bielliptic curve if and only if it appears below:
\begin{enumerate}
\item It is a pair $(N,W_N)$ such that $|W_N|=2$ and $N$ is in the
set
$$\{40,48,52,63,68,72,75,76,80,96,98,99,100,108,124,188\},$$
or it is a pair  $(N,W_N)$ such that $|W_N|=4$ and $N$ is in the
seet
$$\{84,90,120,126,132,140,150,156,220.\},$$
 All these quotient modular curves are bielliptic over $\Q$ with
an elliptic quotient given by $X_0^*(N)$ of genus 1,
\item  or it is one of the following $29$ pairs, ordered by genus
$$\begin{array}{|c|l|}
\hline Genus&(N,W_N)\\ \hline \hline
2&(44,\langle w_4\rangle),(60,\langle w_{20}\rangle),(60,\langle w_4,w_3\rangle)\\
\hline 3&(56,\langle w_8\rangle),(60,\langle w_4\rangle)\\
\hline
4&(60,\langle w_3\rangle),(60,\langle w_5\rangle),(112,\langle w_7\rangle),(168,\langle w_3,w_{56}\rangle)\\
\hline
5&(84,\langle w_4\rangle),(88,\langle w_{11}\rangle),(90,\langle w_9\rangle)\\
&(117,\langle w_9\rangle),(120,\langle w_{15}\rangle ),(126,\langle w_{63}\rangle),(168,\langle w_8,w_7\rangle),\\
&(168,\langle w_7,w_{24}\rangle),(180,\langle w_4,w_9\rangle),(184,\langle w_{23}\rangle),(252,\langle w_4,w_{63}\rangle)\\
\hline
6&(104,\langle w_8\rangle),(168,\langle w_8,w_3\rangle)\\
\hline
 7&(120,\langle w_{24}\rangle),(136,\langle w_8\rangle),(252,\langle w_9,w_7\rangle)\\
 \hline\hline
 9&(126,\langle w_9\rangle),(171,\langle w_9\rangle),(252,\langle w_4,w_9\rangle)\\
 \hline
 10&(176,\langle w_{16}\rangle)\\ \hline

\end{array}$$
\end{enumerate}
\end{teo}

\begin{rem}
It turns out that almost all of the curves listed in Theorem
\ref{mainpq} are bielliptic over $\QQ$. The only exceptions are the
isomorphic curves $X_0(126)/w_{63}$ and $X_0(252)/\langle w_4,
w_{63}\rangle$ which are bielliptic over $\QQ(\sqrt{-3})$.
\par
For more information see the tables in Appendix B where we list
bielliptic involutions and the conductor of the corresponding
elliptic quotient curve as well as the splitting of the Jacobian of
$X_0(N)/W_N$.

\end{rem}
The hyperelliptic $X_0(N)/W_N$ are already known (see \cite{FM}).
Checking which of the non-hyperelliptic ones satisfy the condition
from the theorem discussed at the beginning of this Introduction we
obtain our second main result.

\begin{teo}\label{main2pq}Let $N>1$ be a non square-free integer not a power of a prime.
Assume that the genus of the modular curve $X_0(N)/W_N$ is at least
$2$ for a non-trivial subgroup $W_N$ of $B(N)$ different from
$\langle w_N\rangle$. Then the set
$$\Gamma_2(X_0(N)/W_N,\mathbb{Q})$$ is infinite if and only if
$(N,W_N)$ appears in the following list:
\begin{enumerate}

\item It is a pair $(N,W_N)$ that is an hyperelliptic curve, determined in
\cite{FM} that we reproduce for the convenience of the readers:

$$\begin{array}{|c|l|}
\hline Genus& (N,W_N)\\
\hline\hline 2& (40,\langle w_8\rangle),(40,\langle
w_5\rangle),(44,\langle w_4\rangle),(48,\langle
w_{16}\rangle),(48,\langle w_3\rangle),(52,\langle
w_4\rangle),(54,\langle w_2\rangle),
\\
&(60,\langle w_{20}\rangle),(60,\langle w_{4},w_3 \rangle),(60,\langle w_5,w_{12}\rangle),(72,\langle w_8\rangle),(84,\langle w_4,w_3\rangle ),(84,\langle w_4,w_{21}\rangle)\\
&(84,\langle w_3,w_{28}\rangle),(84,\langle w_7,w_{12}\rangle),(90,\langle w_9,w_5\rangle),(90,\langle w_9,w_{10}\rangle),(90,\langle w_2,w_{45}\rangle)\\
&(90,\langle w_5,w_{18}\rangle),(100,w_4),(120,\langle w_8,w_{15}\rangle),(120,\langle w_{24},w_{40}\rangle),(126,\langle w_2,w_{63}\rangle)\\
&(126,\langle w_{18},w_{14}\rangle),(132,\langle w_4,w_{11}\rangle),(140,\langle w_4,w_{35}\rangle),(150,\langle w_6,w_{50}\rangle)\\
&(156,\langle w_4,w_{39}\rangle)\\
\hline
3&(56,\langle w_8\rangle ),(60,\langle w_4\rangle),(60,\langle w_{60}\rangle),(63,\langle w_9\rangle),(72,\langle w_9\rangle),(120,\langle w_5,w_{24}\rangle),\\
&(126,\langle w_9,w_7\rangle),(126,\langle w_9,w_{14}\rangle),\\
\hline 4&(60,\langle w_{12}\rangle),(168,\langle w_{24},w_{56}\rangle),\\
\hline
5&(92,\langle w_4\rangle).\\
\hline
\end{array}$$

\item or for bielliptic curves that are not hyperelliptic
corresponding to a pair $(99,W_N)$ with $|W_N|=2$.

\end{enumerate}
\end{teo}

We observe that among the elliptic curves $X_0^*(N)$ the $38$ ones
with square-free $N$ all have positive rank, whereas the $25$ ones
with non-square-free $N$ with the exception of $X_0^*(99)$ all have
rank $0$.

\par

We used also codes implemented in \cite{mathematica} and
\cite{magma} for obtaining and supporting above results. Such codes
are available to math community at Quotient Modular Curves folders
in

\begin{verbatim}
https://github.com/FrancescBars
\end{verbatim}

\section{Notation}\label{2}
Let $N>1$ be an integer. We fix once and for all the following
notation.
\begin{itemize}
\item [(i)] We denote by $B(N)$ the group of the Atkin-Lehner involutions of $X_0(N)$.  So, the order of $B(N)$ is $2^{\omega(N)}$, where
    $\omega(N)$ is the number of different primes dividing $N$.

  \item [(ii)]  For $N'|N$, with $(N',N/N')=1$, $B(N')$ denotes the subgroup of $B(N)$ formed by the Atkin-Lehner involutions $w_d$ such that $d|N'$ and $(d,N/N')=1$.
In general, $W_N$ denotes a non-trivial subgroup of $B(N)$.

\item [(iii)] The integers
 $g_N$, $g_{W_N}$ and $g_N^*$ are the genus of $X_0(N)$, $X_0(N)/W_N$ and $X_0^*(N)=X_0(N)/B(N)$ respectively.

\item[(iv)]
 We denote by $\New_N$ the  set of normalized newforms in $S_2(\Gamma_0(N))$. The sets  $\New_N^{W_N}$ and  $S_2(N)^{W_N}$ are the subsets of
 $\New_N$ and  $S_2(\Gamma_0(N))$ formed by the cusp forms   invariant under the action of the group  $W_N$.

\item[ (v)]   $J_0(N)$ and $J_0(N)^{W_N}$ are the  Jacobians  $\Jac(X_0(N))$ and $\Jac(X_0(N)/W_N)$ respectively.

\item [(vi)]Let  $h\in S_2(\Gamma_0(N))$ be an eigenform  of the form $\sum_{d|N/M} c_d f(q^d)$ for some $f\in\New_M$ with $M|N$ and $c_d\in\Z$.
Since for every divisor $d$ of $N/M$ there is a morphism $B_d$ from
$J_0(M)$ to $J_0(N)$ defined over $\Q$ sending every cusp form $g\in
S_2(M)$ to $ g(q^d)\in S_2(N)$, the morphism $\sum_{d|N/M}c_d B_d$
provides  an abelian variety  $A_h$ defined over $\Q$ attached to
$h$ and $\Q$-isogenous to the abelian variety $A_f$ attached by
Shimura to $f$. This abelian variety can be defined as the optimal
quotient of $J_0(N)$ such that the pullback of $\Omega^1_{A_h/\Q}$
is the vector space generated by the Galois conjugates  of
$h(q)\,dq/q$ with rational $q$-expansion.
 This definition determines the $\Q$-isomorphism class of $A_h$, although we are only interested in its $\Q$-isogeny class.

\item[ (vii)] Given two abelian varieties $A$ and $B$ defined over the number field $K$, the notation $A\stackrel{K}\sim B$ stands for $A$ and $B$
    are isogenous over $K$.

\item [(viii)] For an integer  $m\geq 1$ and $f\in\New_N$, $a_m(f)$  is the $m$-th Fourier coefficient of $f$.

\item[(ix)] As usual, $\psi$ denotes the Dedekind psi function. That is, $ \psi(N)=N \prod_{p|N}(1+p^{-1})$, where the product is extended to all primes $p$ dividing $N$.

\item[(x)] We write $\#(w,X)$ for the number of fixed points of the
automorphism $w$ on the curve $X$.

\item[(xi)] We write $S_d=\left(\begin{array}{cc}
1&1/d\\
0&1\\
\end{array} \right)$, and by $w_d^{(kd)}\in GL_2(\mathbb{Z})$ a matrix which
corresponds to a lifting on level $kd$ of the Atkin-Lehner
involution $w_d$ of $X_0(d)$ with $d\geq 2$ and $k\geq 2$ integers,
(thinking $w_d\in GL_2(\Z)$ with determinant $d$).

\item[(xii)] To be able to simultaneously describe the Atkin-Lehner involutions
of different curves (for example in the tables in Appendix A) we use
the notation $\varpi_i=w_{{p_i}^{e_i}}$ where $N=\prod_{i=1}^s
{p_i}^{e_i}$ and $p_1<p_2<\ldots<p_s$ are the different primes
dividing $N$.

\end{itemize}

Recall, if $X_0(N)/W_N$ is bielliptic, there is an involution $u\in
\Aut (X_0(N)/W_N)$, called bielliptic involution, which is unique if
$g_{W_N}\geq 6$ \cite{SiHa}) such that $(X_0(N)/W_N)/u$ is a genus 1
curve defined over a number field $K$ (over $\mathbb{Q}$ if $u$ is
unique). Since $X_0(N)/W_N(\mathbb{Q})$ is not empty, the genus 1
curve has a rational point and, therefore, it is an elliptic curve
$E$ over $K$,  called a bielliptic quotient of $X_0(N)/W_N$. In
particular, such an elliptic curve is a ${K}$-isogeny factor for
$J_0(N)^{W_N}$.

\section{Selecting candidate bielliptic curves $X_0^*(N)$}\label{3}

The starting point of our selection is based on the following
 result, which follows from
\cite[Proposition 1]{SiHa} by considering the natural projection map
$X_0(N)/W_N\rightarrow X_0^*(N)$.
\begin{lema}\label{lemab3.1} Let $N$ be an integer, and take
$X_0(N)/W_N$ of $g_W\geq 2$. If it is bielliptic, then $X_0^*(N)$ is
bielliptic, hyperelliptic or has genus at most one.
\end{lema}

We assume once and for all $N$ not square-free and $N$ not a power
of a prime.

\begin{teo}\label{levelstudy2} Consider $X_0(N)^*$ with $N$ a non-square free level not a power of a prime,
then
\begin{enumerate}
\item (Gonz\'alez-Lario, \cite{GL}) $g_N^*=0$ if and only if $N\in\{12,18,20,24,28,36,44,45,50,54,$\\$ 56,60,92\}$
\item (Gonz\'alez-Lario, \cite{GL}) $g_N^*=1$ if and only if $N\in\{40,48,52,63,68,72,75,76,80,84,90,$\\$ 96,98, 99,100,108,120,124,126,132,140,150,156,188,220\}$
\item (Hasegawa, \cite{Hata}) $g_N^*=2$ if and only if $N\in\{88,104,112,116,117,135,147,153,168,180,$\\$184,198,204,276,284,380\}$
\item (Hasegawa, \cite{Ha97}) $g_N^*>2$ and $X_0^*(N)$ is hyperelliptic if and only if $N$ appears next:
\begin{center}
\begin{tabular}{c|r}
\hline $g_N^*$&$N$\\
\hline 3&136; 171; 207; 252; 315;\\
\hline 4&176;\\
\hline 5&279.\\
\hline
\end{tabular}
\end{center}
\item (Bars-Gonz\'alez, \cite{BaGon2}) $X_0^*(N)$ is bielliptic if and only if $N$ appears in the following table
\begin{center}
\begin{tabular}{c|r}
\hline $g_N^*$&$N$\\
\hline
2& 88; 112; 116; 153; 180; 184; 198; 204; 276; 284; 380;\\
\hline
 3& 144; 152; 164; 189; 196; 207; 234; 236; 240; 245; 248;
252; 294; 312;\\
 & 315; 348; 420; 476;\\\hline
4&148; 160; 172; 200; 224; 225; 228; 242; 260; 264; 275; 280; 300;
306; 342;\\
\hline 5& 364; 444; 495;\\
\hline 7&558.\\
\hline
\end{tabular}
\end{center}
\end{enumerate}
\end{teo}

\section{Selecting possible bielliptic quotient curves}

\subsection{General criteria to discard bielliptic curves}

Here we collect some general criteria that allow to prove
comparatively easily that certain curves are not bielliptic.

\begin{lema} \label{4.4} (special form of the Castelnuovo inequality
\cite[Theorem 3.5]{Accola}) Let $\phi: X\to Y$ be a morphism of
degree $d$ of curves. If $X$ has a bielliptic involution $v$, then
$$g(X)\leq dg(Y)+d+1$$
or the morphism $\phi$ factors over $X/v$.
\par
In particular: A hyperelliptic curve of genus $g\geq 4$ cannot be
bielliptic. A trigonal curve of genus strictly bigger than $4$
cannot be bielliptic. A curve of genus $g\geq 6$ has at most one
bielliptic involution.
\end{lema}

\begin{prop}\label{centraloverK}\cite[Proposition 3.2]{JKS}
Let $X$ be a bielliptic curve of genus $g\geq 6$ defined over a
field $K$ of characteristic $0$. Then the bielliptic involution is
unique, defined over $K$ and lies in the center of $\Aut (X)$.
\end{prop}

\begin{proof}
The uniqueness was just mentioned in Lemma \ref{4.4}. Acting on this
bielliptic involution with $\Aut (X)$ by conjugation resp.  with
$\Gal (\overline{K}/K)$,  the uniqueness implies the other
properties.
\end{proof}

The following Lemmas \ref{4.1} and \ref{4.2} appeared in \cite{Sc}
and \cite{JKS} respectively for genus $\geq 6$. Here we present a
slightly different proof for general genus.
\begin{lema}\label{4.1} Let $w$ be an involution of $X$ with more than 8 fixed
points. Then either $w$ is a bielliptic involution or $X$ is not
bielliptic
\end{lema}
\begin{proof}Let $g$ and $h$ be the genera of $X$ and $X/w$. Since $w$ has more than 8 fixed points, by the Hurwitz formula we have $g>2h+3$. If $v$ is a bielliptic involution,
 from Castelnuovo we get the contradiction that $g$ cannot be bigger than $2h+3$, unless the two involutions factor over a common curve, i.e. are the same.
\end{proof}

\begin{lema}\label{4.2}\label{unramifiedcovering}
Let $X$ be a curve of genus $g$ with a bielliptic involution $v$ and
let $G$ be a subgroup of $Aut(X)$ such that the curve $Y=X/G$ has
genus $h\geq 2$.
\begin{itemize}
\item[{(a)}] If the map $\phi:X\to Y$ is ramified, i.e. if $g-1>|G|(h-1)$, and
$g\geq 6$, then $Y$ must be hyperelliptic and $v$ induces the
hyperelliptic involution on $Y$.
\item[{(b)}] (unramified covering criterion) If $Y$ is not hyperelliptic, then it must be bielliptic and the map $\phi:X\to Y$ must be unramified, i.e. $g-1=|G|(h-1)$.
\end{itemize}

\end{lema}

\begin{proof}

(a) Obviously $v\not\in G$ because $h>1$. Since $g\geq 6$, we know
from Proposition \ref{centraloverK} that $v$ is central. So $v$
induces an involution $\widetilde{v}$ on $Y$, and $G$ induces a
group $\widetilde{G}$ (isomorphic to $G$) of automorphisms on $X/v$.
If $\phi$ is ramified, at least one nontrivial element of $G$ has at
least one fixed point. So the same holds for $\widetilde{G}$. Hence
the Hurwitz formula for the covering from $X/v$ to
$(X/v)/\widetilde{G}=Y/\widetilde{v}$ shows that the latter curve
has genus $0$.
\par
(b) Since $Y$ is not hyperelliptic we have $h\geq 3$. Assume that
$\phi$ is ramified. Then the Hurwitz formula implies $g\geq 6$, and
part (a) leads to a contradiction.

\end{proof}

If $X_0(N)/W_N$ is bielliptic over $\QQ$ and $p\nmid N$, then the
morphism reduces modulo $p$ and we have that
\begin{equation}\label{eq4.1}
|X_0(N)/W_N(\mathbb{F}_{p^n})|\leq 2 |E(\mathbb{F}_{p^n})|
\end{equation}
for all $n\geq 1$, which can be computed. In particular we reproduce
similar results as \cite[Lemmas 5,7]{BaGon}:

\begin{lema}\label{lemadisca} Assume $X_0(N)/W_N$ is bielliptic over $\mathbb{Q}$, and
$p\nmid N$. Then the following equality holds: $
\displaystyle{\frac{\psi (N)}{|W_N|}\leq 12\cdot
\frac{2|E(\mathbb{F}_{p^2})|-1}{p-1}}\,$.
 In
particular we have, $
 \displaystyle{\frac{\psi (N)}{|W_N|}}\leq 12\cdot
\frac{2(p+1)^2-1}{p-1}\,$.

\end{lema}
\begin{proof}
Assume $p\nmid N$. We generalize the argument used by Ogg in
\cite{Ogg}. Indeed,  $X_0(N)(\F_{p^2})$ contains $2^{\omega(N)}$
cusps and at least $(p-1)\frac{\psi(N)}{12}$ {{many}} supersingular
points (cf. \cite[Lemma 3.20 and 3.21]{BGGP}). Since there is a
nonconstant morphism defined over $\Q$ from $X_0(N)$ to an elliptic
quotient $E$ of $X_0(N)/W$ which has degree $2\cdot |W_N|$, we get
$|X_0(N)(\F_{p^2})|\leq 2\cdot|W_N| |E(\F_{p^2})|$.
\end{proof}

Similarly, for optimal quotients with conductor $M=N$ for elliptic
quotient (related with the strong Weil parametrization) it is easy
to derive the following lemma.

\begin{lema}\label{lemadegree}
Let $E'$ be the optimal elliptic curve in the $\Q$-isogeny class of
the bielliptic quotient $E$  of $X_0(N)/W_N$ with conductor $M=N$.
Then the degree $D$  of the modular parametrization $\pi_N \colon
X_0(N)\rightarrow E'$ divides $2\cdot |W_N|$. Note that the degree
$D$ can be found in \cite[Table 5]{Cre}.
\end{lema}

\begin{teo}\label{cyclicdisjoint}
Let $X$ be a curve over a field of characteristic $0$.  Then
\begin{itemize}
\item[(a)] The stabilizer in $\Aut (X)$ of any point of $X$ is cycylic.
\item[(b)] Distinct involutions in $\Aut (X)$ have disjoint fixed points.
\end{itemize}
\end{teo}

\begin{proof}
(a) is proved in \cite[Corollary III.7.7]{FK}, and (b) follows
immediately from (a).
\end{proof}

\begin{prop}\label{prop3.1} Let $X$ be a curve of genus $g$ at least $6$. Assume that $\Aut (X)$ has a
subgroup $H$ of order $2^t$ such that $2^t$ does not divide
$2(g-1)$. Then either the bielliptic involution of $X$ is contained
in $H$ or $X$ is not bielliptic.
\end{prop}
\begin{proof} Let $v$ (outside $H$) be the bielliptic involution. Since $v$ commutes
with the elements of $H$, $H$ acts on the $2(g-1)$ fixed points of
$v$. If $|H|$ does not divide $2(g-1)$, there must be an orbit whose
length is not $|H|$. So there must be an involution in $H$ that
fixes a point in this orbit. But $v$ also fixes this point,
contradicting Theorem \ref{cyclicdisjoint}.
\end{proof}
 The curves $X_0(N)/W_N$ always have a subgroup $H$ (isomorphic to
$B(N)/W_N$), and we can discard directly even genus quotient curves
when $B(N)/W_N$ is of order at least 4.

\subsection{Criteria for genus 5 bielliptic curves}

For curves of genus $g\leq 5$ it can be difficult to decide whether
they are bielliptic, because there might be several bielliptic
involutions, but none of them is guaranteed to be defined over $\Q$.
The following results will be helpful for several of the curves we
will encounter.

\begin{lema}\cite[p.50]{Accola} A genus $5$ curve that is a degree two covering of a
hyperelliptic genus $3$ curve is either hyperelliptic or bielliptic.
\end{lema}

\begin{lema}\cite[Lemma 2.3]{KMV} \label{g5biell}
Let $X$ a bielliptic curve of genus $5$. Then it has $1$, $2$, $3$
or $5$ bielliptic involutions, and these involutions commute and
generate a group of exponent $2$ and order $2$, $4$, $8$ and $16$
respectively. Moreover, the product of any two bielliptic
involutions is an involution whose quotient curve is of genus $3$
and hyperelliptic.
\end{lema}

\begin{lema}\label{lemgenus2}
Let $X$ be a bielliptic genus $5$ curve with an involution $u$, such
that $X/u$ has genus $2$.
\begin{itemize}
\item[(a)]  If $X$ has an odd number of bielliptic involutions, then there
is a unique one among them, call it $v$, that commutes with $u$. If
moreover $u$ is defined over $\Q$, then $v$ is also defined over
$\Q$.
\item[(b)]  If $X$ has exactly $2$ bielliptic involutions, neither one of
them will commute with $u$.
\end{itemize}
\end{lema}

\begin{proof}
$u$ acts on the bielliptic involutions by conjugation, so they come
in pairs of conjugated ones plus the ones which commute with $u$.
This already shows that there is at least one such $v$ (resp. two or
none if $X$ has exactly 2 bielliptic involutions).
\par
Now we show that there can be at most one such $v$. Since $v$
commutes with $u$ and the map $X\to X/u$ is ramified, we can argue
as in the proof of Theorem \ref{unramifiedcovering} (a) with the
role of $G$ being played by $\langle u\rangle$. Again we obtain that
$v$ must induce the hyperelliptic involution on $X/u$. So the only
possibility for a bielliptic involution other than $v$ that commutes
with $u$ would be $uv$. But they cannot both be bielliptic, because
then $u=(uv)v$ would have quotient genus $3$ by Lemma \ref{g5biell}.
\par
If $u$ is defined over $\Q$ and commutes with $v$, then from the
action of the absolute Galois group of $\Q$ it also commutes with
all Galois conjugates of $v$, which by the uniqueness must be equal
to $v$.
\end{proof}

\begin{lema} Let $X$ be a curve of genus $5$ with an involution $w$ such
that $X/w$ is of genus $3$ and non-hyperelliptic. If $X$ is
bielliptic, then $X$ has exactly one or exactly $3$ bielliptic
involutions.
\end{lema}

\begin{proof} First we treat the case where $X$ has 5 bielliptic involutions.
Then by \cite[Lemma 2.3 and Remark 3.2]{KMV} the bielliptic
involutions generate a group $H$ of order $16$ in which all other
elements are involutions which have a quotient that is hyperelliptic
of genus $3$. In particular, $w$ is not in $H$.
      Pick a bielliptic involution $v_1$ that does
not commute with $w$ and define $v_2$ as its conjugate under $w$. If
$w$ commutes with all bielliptic involutions, just pick any two
$v_1$ and $v_2$. In either case $w$ commutes with the involution
$v_1v_2$, which has quotient genus 3. By a classical result (see for
example \cite[Lemma 5.10]{Accola}) a genus $3$ curve that is a
degree $2$ cover of a genus $2$ curve must be hyperelliptic. So
since $X/w$ is not hyperelliptic, $X/\langle v_1v_2,w\rangle$ cannot
have genus 0 or 2, and hence has genus 1. From
$$g(X/w)+g(X/v_1v_2)+g(X/v_1v_2w)=g(X)+2g(X/\langle v_1v_2,w\rangle)=7$$
(see for example \cite[p.49]{Accola}) we see that $X/v_1v_2w$ also
has genus 1. Therefore $v_1v_2w$ is in $H$. With $v_1v_2$ in $H$ we
get the contradiction $w\in H$. So we have proved that $X$ has at
most 3 bielliptic involutions.
\par
Now assume that $X$ has exactly 2 bielliptic involutions. Call them
$v_1$ and $v_2$.  By the same argument as before we get a third
bielliptic involution.
\end{proof}

\begin{cor} Let $X$ be a curve of genus 5 over $\mathbb{Q}$ that has
involutions $u$ and $w$ as in the two lemmas. with $u$ defined over
$\mathbb{Q}.$ If $X$ is bielliptic, then it has one or three
bielliptic involutions and at least one bielliptic involution is
defined over $\mathbb{Q}$.
      More precisely, if it has 3 bielliptic
involutions, then there is a unique one that commutes with $u$, and
that one is defined over $\mathbb{Q}$. (The other two might be
defined over $\mathbb{Q}$ or not.)
\end{cor}

\subsection{Formulas for the genus of a quotient curve by certain involutions}

The most natural way to prove that a curve is bielliptic is to
exhibit a bielliptic involution. We will do this for practically all
the bielliptic curves $X_0(N)/W_N$. The most obvious candidates are
the Atkin-Lehner involutions outside $W_N$. If $N$ is divisible by
$4$ or $9$, then $X_0(N)$ has additional involutions, which under
certain conditions induce involutions on $X_0(N)/W_N$. See below for
details.
\par
When $v$ is an involution of $X_0(N)$ that induces an involution
$\widetilde{v}$ on $X_0(N)/W_N$, it is usually quicker not to bother
about the fixed points of $\widetilde{v}$ on $X_0(N)/W_N$ but to
determine the genus of $(X_0(N)/W_N)/v$ from the fixed points of the
elements of $G=\langle W_N, v\rangle$ on $X_0(N)$ by applying the
Hurwitz formula to the covering $X_0(N)\to X_0(N)/G$.
\par

In the situations we encounter the group $G$ is of the form
$G\cong\mathbb{Z}/2\times\cdots\times\mathbb{Z}/2$.  Then by Theorem
\ref{cyclicdisjoint} the stabilizer of a point is trivial or
$\mathbb{Z}/2$. Moreover, all involutions in $G$ have disjoint fixed
points, and the Hurwitz formula takes the following simple shape.

\begin{lema}\label{4.13}
Let $G$ be a subgroup of $Aut(X_0(N))$ in which all the non-trivial
elements are involutions. Then the fixed points of these involutions
are disjoint and the genus of $X_0(N)/G$ is obtained by the formula
$$|G|(2g(X_0(N)/G)-2)+
\sum_{w\in G}\#(w, X_0(N)) = 2g(X_0(N))-2.$$
\end{lema}

Formulas for the number of fixed points of the Atkin-Lehner
involutions on $X_0(N)$ are in \cite{Ogg}. For the other involutions
we will review and expand the known results in this section.
\par
Recall $S_k = \left(\begin{array}{cc} 1&1/k\\ 0& 1\\
\end{array}\right)$. The following result is well-known (see \cite{FM} or
\cite{Ba1}).

\begin{prop}\label{prop1S4.3} Let $N=2^\alpha M$ with $\alpha\geq 2$ and $M$
odd.

\begin{itemize}

\item[{(a)}]   Then
$S_2$ is an involution of $X_0(N)$, defined over $\QQ$, and commutes
with all Atkin-Lehner involutions $w_r$ for which $r$ is odd. Hence,
$V_2 =S_2 w_{2^\alpha}S_2$ also is an involution of $X_0(N)$,
defined over $\QQ$, and commutes with all $w_r$ for which $r||M$.

\item[{(b)}]   If $\alpha\geq 3$, then $V_2$ also commutes with $w_{2^\alpha}$. So $V_2 w_{2^\alpha}$ is an involution, and consequently $S_2 w_{2^\alpha}$ has order $4$. In fact,
$\langle S_2, w_{2^\alpha}\rangle\cong D_4$.

\item[{(c)}]  If $\alpha=2$, then
$\langle S_2, w_4\rangle$ is non-abelian of order $6$ with $V_2 =S_2
w_4 S_2=w_4 S_2 w_4$ being the third involution and $S_2 w_4$ and
$w_4 S_2$ having order $3$.

\end{itemize}
\end{prop}

\begin{lema}\label{lem2S4.3} \cite[Proposition 3.5]{Ba1} If $N=2^\alpha M$ with $\alpha\geq
2$ and $M$ odd, then
$$X_0(N)/w_{2^\alpha}S_2 w_{2^\alpha} = X_0(N/2).$$
\end{lema}

\begin{proof} An easy calculation shows that $w_{2^\alpha}S_2 w_{2^\alpha}$
lies in $\Gamma_0(N/2)$ but not in $\Gamma_0(N)$.
\end{proof}

\begin{lema}\label{3lemS4.3} Let $u$ and $v$ be two commuting involutions on a curve $X$.
Then $uv$ is also an involution and
$$\#(uv, X)=
2\#(u, X/v)-\#(u, X).$$
\end{lema}
\begin{proof}
By Theorem \ref{cyclicdisjoint} , the fixed points of the three
involutions $u$, $v$ and $uv$ are disjoint, and each fixed point has
ramification index $2$ in the degree $4$ covering $X\to X/\langle u,
v\rangle$.
\par
This implies that every fixed point of $u$ or $uv$ must be ramified
in $X/v\to X/\langle u, v\rangle$. Conversely, if a point of $X/v$
is ramified in $X/v\to X/\langle u, v\rangle$, then the two points
of $X$ lying above it must have ramification group $\langle
u\rangle$ or $\langle uv\rangle$. So, all in all $\#(u, X)+\#(uv,
X)=2\#(u, X/v)$.
\end{proof}

\begin{lema} \label{4lemS4.3} Let $N=2^\alpha M$ with $\alpha\geq 2$ and $M$ odd.
Also let $r||M$.

\begin{itemize}

\item[{(a)}]   \cite[Proposition 3.9]{Ba1}
$\#(V_2, X_0(N))= \#(w_{2^\alpha}, X_0(N))$ and \newline $\#(V_2
w_r, X_0(N))= \#(w_{2^\alpha}w_r, X_0(N))$.

\item[{(b)}]  $\#(S_2, X_0(N))=
\#(w_{2^\alpha}S_2 w_{2^\alpha}, X_0(N))=
(2g(X_0(N))-2)-2(2g(X_0(N/2))-2)$.

\item[{(c)}]   \cite[Proposition 3.6]{Ba1}
$\#(S_2 w_r, X_0(N))= \#(w_{2^\alpha}S_2 w_{2^\alpha}w_r, X_0(N))=$
\newline
$2\#(w_r, X_0(N/2))-\#(w_r, X_0(N))$.

\item[{(d)}]    \cite[Proposition 3.10]{Ba1}
If $\alpha\geq 3$, then $$\#(V_2 w_{2^\alpha}, X_0(N))= 2\#(S_2,
X_0(N/2))-\#(S_2, X_0(N))\; and$$  $$\#(V_2 w_{2^\alpha}w_r,
X_0(N))= 2\#(S_2 w_r, X_0(N/2)) -\#(S_2 w_r, X_0(N)).$$

\end{itemize}

\end{lema}

\begin{proof} Note that conjugate involutions have the same number of fixed
points. This proves (a) and together with Lemma \ref{lem2S4.3} also
(b). The other formulas are special cases of Lemma \ref{3lemS4.3},
also using Lemma \ref{lem2S4.3}.
\end{proof}

\begin{lema}\label{4.19}\label{newlemma}\label{lemmaV3}\label{lemaV3}

Let $9||N$ and $S_3 =\left(\begin{array}{cc} 1&1/3\\
0&1\end{array}\right)$.

\begin{itemize}

\item[{(a)}]    $S_3$ normalizes $\Gamma_0(N)$ and induces an automorphism of $X_0(N)$ of order $3$ defined over $\QQ(\sqrt{-3})$. Its Galois conjugate is $S_3^2$. Moreover, $S_3$ commutes with the Atkin-Lehner involutions $w_r$ with
$r\equiv 1\ \mod\ 3$, whereas for $r\equiv 2\ \mod\ 3$ we have $w_r
S_3=S_3^2 w_r$ and $w_9 S_3$ has order $3$.

\item[{(b)}]    $V_3 =S_3 w_9 S_3^2$ is an involution of $X_0(N)$. With respect to Atkin-Lehner involutions we have
$w_r V_3=\left\{\begin{array}{cc} V_3 w_r &  if\; r\equiv\ 1\ \mod \
3\; or\; r=9\; and\\
 V_3 w_9 w_r&  if\;  r\equiv\ 2\ \mod \ 3\\
 \end{array}\right.$

Moreover, if $r\equiv\ 2\ \mod\ 3$ then $\langle V_3,
w_r\rangle\cong D_4$ and $V_3 w_r$ has order $4$ with $(V_3 w_r)^2
=w_9$.

\item[{(c)}]     $V_3$ as an involution of $X_0(N)$ is defined over $\Q(\sqrt{-3})$. Its $Gal(\QQ(\sqrt{-3})/\QQ)$-conjugate is $V_3 w_9$. In particular, $V_3$ and
$V_3 w_9$ have the same number of fixed points on $X_0(N)$.

\item[{(d)}]    More generally we have
$$\#(V_3 w_9, X_0(N))=\#(V_3, X_0(N))=
\#(w_9, X_0(N))$$ and for $r\equiv\ 1\ \mod\ 3$ also
$$\#(V_3 w_9 w_r, X_0(N))=
\#(V_3 w_r, X_0(N))= \#(w_9 w_r, X_0(N)).$$

\item[{(e)}] $V_3$ as an involution of $X_0(N)/W$ is defined over $\Q$ if and only if $w_9\in W$.

\end{itemize}
\end{lema}
\begin{proof}Most of this (and some more) is already in \cite{Ba1}, \cite{FM} or \cite{Ba2}. So we only prove (c).
\par
$V_3 =S_3 w_9 S_3^2$ is defined over $\QQ(\sqrt{-3})$ because $S_3$
is. Now let $\sigma$ be the nontrivial Galois automorphism of
$\QQ(\sqrt{-3})$. Then $\sigma(V_3)= \sigma(S_3 w_9 S_3^2)=S_3^2 w_9
S_3$. So $V_3\sigma(V_3)=w_9(w_9 S_3)^3 =w_9$ by part (a).
\end{proof}
\begin{rem} From Lemma \ref{4.19}(e), we obtain the result in \cite{BaGon2} that $V_3$ as an involution of $X_0^*(N)$ is always defined over
$\Q$.
\end{rem}
\subsection{Useful isomorphisms on non-square free quotient curves}

For later use we refine \cite[Proposition 4]{FM}.

\begin{prop}\label{4.10}
 Suppose $4||N$ and write $N=4M$. Let $W'$ be a subgroup of $B(N)$
generated by $w_4$, $w_{m_1},\ldots, w_{m_s}$ with $m_i ||M$. Then
we have
$$X_0(N)/W'\cong
X_0(N)/\langle S_2 w_4 S_2, w_{m_1},\ldots,  w_{m_s}\rangle=
X_0(N)/\langle w_4 S_2 w_4, w_{m_1},\ldots,w_{m_s}\rangle=$$
$$X_0(2M)/\langle w_{m_1}, \ldots, w_{m_s}\rangle.$$

Hence if $A\in GL_2(\mathbb{R})$ is a bielliptic involution of
$X_0(2M)/\langle w_{m_1},\ldots, w_{m_s}\rangle$, then $S_2 AS_2$
normalizes $\langle \Gamma_0(N), W'\rangle$ and induces a bielliptic
involution on $X_0(N)/W'$.
\end{prop}
\begin{proof}
The isomorphism comes from conjugating with $S_2$, the first
equality from part (c) of Proposition \ref{prop1S4.3} and the final
equality from Lemma \ref{lem2S4.3}.
\end{proof}

\begin{prop}\label{4.12}\cite[Proposition 5]{FM}
Assume $9||N$. Let $W'$ be a subgroup of $B(N)$ generated by
$w_{n_1},\ldots, w_{n_t}$ ($n_i ||N$) and let $W''= \langle
\{w_{n_i} w_9^{e(n_i)}\}_{i\in\{1,\dots, t\}}\rangle$ where $e(m)=0$
if $m\equiv\ 1\ \mod\ 3$ or if $9||m$ and $m/9\equiv\ 1\ \mod\ 3$,
and $e(m)=1$ otherwise. Then $V_3$ induces an isomorphism
$$X_0(N)/W'\cong X_0(N)/W''.$$

\end{prop}

\section{Jacobian decomposition, field of bielliptic involutions,
Petri theorem.}

\subsection{On the field where bielliptic involutions may be
defined}

Let $X_0(N)/W_N$ be a quotient curve. We want to control if the
automorphism, or more concretely if a candidate to bielliptic
involution is defined over $\Q$ or a number field. In order to
control the number field $K$ (when $g_{W_N}\leq 5$), we have the
following results in \cite{BaGon2}:

\begin{prop}\label{defi}
Let $A$ be a modular  abelian variety defined over $\Q$ such that
$A\stackrel{\Q}\sim \prod_{i=1}^m A_{f_i}^{n_i}$ for some
$f_i\in\New_{N_i}$, where $A_{f_i}$ are pairwise non-isogenous over
$\Q$. All endomorphisms of $A$ are defined over $\Q$ if, and only
if, for every nontrivial  quadratic Dirichlet character $\chi$, the
newform $f_i\otimes \chi$ is different from any  Galois conjugates
of $f_j$ for all $i$ and $j$.
\end{prop}
\begin{rem}
If $\chi$ is the quadratic Dirichlet character attached to the
quadratic number field $K=\Q( \sqrt D)$, then there is an isogeny
between the abelian varieties $A_f$ and $A_{f\otimes \chi}$ defined
over $K$.
\end{rem}

Also, the following result specific for modular forms
non-corresponding to elliptic curves \cite{Pyle} clarifies the
possible elliptic quotient that could appear:
\begin{prop}\label{Pyle} When $\dim A_f>1$ and $f$ does not have complex
multiplication (CM), i.e. $f\neq f\otimes \chi$ for all quadratic
Dirichlet characters, a necessary condition for $A_f$ to have an
elliptic quotient over $\overline{\Q}$ is $a_p(f)^2\in\Z$ (the
$p$-th Fourier coefficient of the modular form $f$) for all primes
$p$.
\end{prop}

\subsection{On the Jacobian decomposition of quotient modular curves}

We  recall that the $\mathbb{Q}$-decomposition for $J_0(N)$ has the
form
$$
J_0(N)\stackrel{\Q}\sim \prod_{M|N} \prod_{f\in \New_M/G_\Q}
A_{f}^{n_f}\,,$$
 where $n_f$ is the number of positive divisors of $N/M$ and $G_\Q$ denotes the absolute Galois group $\Gal(\overline{\Q}/\Q)$. Each  newform
 $f\in\New_M$
provides an $n_f$-dimensional vector subspace of $ S_2(N)$ generated
by $\{ f(q^d)\colon 1\leq d| N/M\}$.

To determine the $\Q$-decomposition for $J_0(N)^{W_N}$
 $$
J_0(N)^{W_N}\stackrel{\Q}\sim \prod_{M|N} \prod_{f\in\New_M/G_\Q}
A_{f}^{m_f}\,,$$ we need to control  which $A_{f}$ appears in this
decomposition and the precise exponent  $0\leq m_f\leq n_f$, and
next results  allow us, once fixed $(N,W_N)$, to determine a basis
of $S_2(N)^{W_N}$ and, in particular, the splitting of
$J_0(N)^{W_N}$ (see \cite[Lemma 2.1, Prop.2.2]{BaGon2}).

\begin{lema}\label{eigenvector}
Let $M$ and $N$ be positive integers such that $ M|N$.  Let  $M_1 $
be a positive divisor of $M$ such that $\gcd(M,M/M_1)=1$ and let $d$
be  a positive divisor of $N/M$ such that
$\gcd(M_1\,d,N/(M_1\,d))=1$.
 If $f\in S_2(\Gamma_0(M))$ is an  eigenvector of the Atkin-Lehner involution $w_{M_1}$ with eigenvalue $\varepsilon (f)$ and $\varepsilon\in
 \{-1,1\}$, then $f(q)+\varepsilon\, d \, f(q^d)\in S_2(\Gamma_0(N))$ is an  eigenvector
  of the Atkin-Lehner involution $w_{M_1\,d}$ with eigenvalue $\varepsilon(f)\cdot \varepsilon$.
\end{lema}

\begin{prop}\label{eigenforms}
Assume that $N=p^k\cdot M$, where $k\geq 1$, $p$ is a prime and $M$
is an integer coprime to $p$. For $0\leq i< k$, let  $f\in
S_2(\Gamma_0(p^i\cdot M))^{W}$ be such that
$w_{p^i}(f)=\varepsilon\cdot f$ with $W\leq B(M)$ (clearly $
\varepsilon= 1$ when $i=0$). Let  $\cS$ be the  vector subspace of
$S_2(\Gamma_0(p^k\cdot M))^W$ generated by the $k-i+1$ linearly
independent, $\mathbb{Q}$-isogenous to $f$, eigenforms $ \{f,
B_p(f),\cdots, B_p^{k-i}(f)\}$. Then,
\begin{itemize}
\item[\rm{(i)}] The following normalized  eigenforms
$$ g_0=(1+ p B_p)^{k-i}f,  \cdots,g_j= (1+ p B_p)^{k-i-j}(1- p B_p)^j f,\cdots, g_{k-i}=(1- p B_p)^{k-i}f\,,$$
are a basis of $\cS$ (recall $B_p$ is the morphism sending a modular
form $g(q)$ to $g(q^p)$).
\item[\rm{(ii)}] Every $g_j$ is an
eigenvector of $w_{p^k}$ with eigenvalue $(-1)^j\varepsilon$.

\end{itemize}
\end{prop}

Consider the $\mathbb{Q}$-decomposition for $J_0(N)^{W_N}$
$$
J_0(N)^{W_N}\stackrel{\Q}\sim \prod_{M|N} \prod_{f\in\New_M/G_\Q}
A_{f}^{m_f}\,.$$
 Now,  for $f\in\New_M$ if $m_f>0$ then $f$ is necessarily fixed by
the Atkin-Lehner involutions $w_d\in W_N$, with $d||M$, and $f$
provides $m_f$-eigenforms $g_i\in S_2(N)^{W_N}$ lying in the vector
space generated by $\{ f(q^d)\colon 1\leq d| N/M\}$. The integer
$m_f$ is determined by  using Lemma \ref{eigenvector} and
Proposition \ref{eigenforms}. Read readme.md file in
\begin{verbatim}
https://github.com/FrancescBars/Magma-functions-on-Quotient-Modular-Curves
\end{verbatim} in order to obtain such decompositions.

Jacobian decomposition allows us to compute
$|X_0(N)/W_N(\mathbb{F}_{p^n})|$ for all $p\nmid N$ thanks to the
Eichler-Shimura congruence (see a MAGMA function code
FpnpointsQuotientCurve in
\begin{verbatim}
https://github.com/FrancescBars/Magma-functions-on-Quotient-Modular-Curves/blob/main/funcions.m
\end{verbatim}
and examples in the Readme.md in such github folder.

\subsection{The application of a result of Petri}

If $X_0(N)/W_N$ is hyperelliptic, we know an equation (see
\cite{FM},\cite{Hata}) and MAGMA computes the automorphism group
over $\mathbb{Q}$. In this case, if  $X_0(N)/W_N$ has a
non-hyperelliptic involution over $\mathbb{Q}$, we can compute the
genus of the quotient curve by using \cite[Proposition 1]{Ogg}
 and determine if $X_0(N)/W_N$ is bielliptic or
not over $\mathbb{Q}$. For finite extensions of $K$ we need to deal
with the decomposition of the Jacobian and study the endomorphism
algebra following Proposition \ref{defi} and to compute it in such
number field, which is a quadratic field. MAGMA computes such
automorphism group over quadratic fields. Thus for hyperelliptic
quotient modular curves we can decide if they are bielliptic or not.

Consider a non-hyperelliptic curve $X$ of genus $g\geq 3$ defined
over a subfield $K$ of the complex field $\C$. For a fixed basis
$\omega_1,\cdots,\omega_g$ of $\Omega_{X/K}^1$ and an integer $i\geq
2$, we denote by $\cL_i$ the $K$-vector space formed by the
homogenous polynomials $Q\in K[x_1,\cdots,x_g]$ of degree $i$ such
that $Q(\omega_1,\cdots,\omega_g)=0$.

By using a theorem of Petri, \cite[Lemma 13]{BaGon} characterizes
the existence of a bielliptic involution of $X_0^*(N)$ with $N$
square-free for non-hyperelliptic curves. Later, \cite[Proposition
2.6]{BaGon2} generalizes this result to any non-hyperelliptic curve
of genus $>2$.

\begin{prop}\label{exponent}\label{2.5} With the above notation, assume that $\Jac(X)\stackrel{K}\sim E^m\times A$, where $E$ is an elliptic curve and $A$ an
abelian variety such that does not have $E$ as a quotient defined
over $K$. Denote by $I_{g-m}\in M_{g-m}(\mathbb{Q})$ the identity
matrix. Take the basis $\{\omega_i\}$ such that $\omega_1,\cdots,
\omega_m$ and $\omega_{m+1},\cdots ,\omega_g$ are bases of the
pullback of $\Omega^1_{E^m/K}$ and $\Omega^1_{A/\mathbb{Q}}$
respectively. Then, $E$ is $K$-isogenous to the Jacobian of a
bielliptic quotient of $X$ over $K$  if, and only if, there exists a
matrix $\cA\in\GL_m(K)$ that satisfies
\begin{equation}\label{cond-inv}
Q((-x_1,x_2,\cdots,x_g)\cdot \cB)\in \cL_i' \text{ for all $Q\in
\cL_i$ and for all $i\geq 2$}\,,
\end{equation}
where $\cB$ is the matrix $\left(\begin{array}{c|c}\cA& 0 \\ \hline
0& I_{g-m} \end{array} \right)\in \GL_g(K)$ and $\cL'_i=\{
Q((x_1,x_2,\cdots,x_g)\cdot \cB))\colon Q\in\cL_i\}$.
\end{prop}

\begin{rem}\label{2.6} The $K$-vector  space $\cL'_i$ is the set of homogenous polynomials in $ K[x_1,\cdots,x_g]$ of degree $i$ such that
$Q(\omega_1',\cdots,\omega_m',\omega_{m+1},\cdots \omega_g)=0$,
where
$(\omega_1',\cdots,\omega_m')=\cA^{-1}(\omega_1,\cdots,\omega_m)$.
\end{rem}
\begin{rem}\label{poli}\label{2.7} We recall that if $g=3$, then $\dim \cL_4=1$ and the condition (\ref{cond-inv}) can be restricted to $i=4$.
 When $g>3$, $\dim \cL_2=(g-3)(g-1)/2$. In this case, it suffices to check  (\ref{cond-inv}) only for $i=2,3$ and,
 in the particular case that $X$ is neither a smooth quintic plane curve ($g=6$) nor a trigonal  curve, we can restrict the condition to $i=2$.

\end{rem}

\vskip 0.2 cm As in \cite{BaGon}, for $j\leq g$ we introduce the
$K$-vector space
$$
\cL_{2,j}=\{ Q\in\cL_2\colon Q(x_1,\cdots,x_{j-1},-
x_j,x_{j+1},\cdots, x_n)\in\cL_2\}\,.
$$
 By using that the polynomials in $\cL_2$ are irreducible, in \cite{BaGon} it is proved that
$$
\cL_{2,j}=\{ Q\in\cL_2\colon Q(x_1,\cdots,x_{j-1},
x_j,x_{j+1},\cdots, x_n)=Q(x_1,\cdots,x_{j-1},- x_j,x_{j+1},\cdots,
x_n)\}\,.
$$
and a modular form corresponding to a dimension one abelian variety
of genus $>3$ with associated differential $\omega_i$, is bielliptic
if and only if $\dim \cL_2=\dim \cL_{2,i}$ (here we are assuming not
a smooth plane quintic curve, nor a trigonal curve).

 A similar result is obtained when $g=3$ and we replace
$\cL_{2,j}$ with $\cL_{4,j}$.
\begin{rem}

We have $J_0(N)^{W_N}\stackrel{\Q}\sim \prod_{i=1}^s A_{f_i}^{n_i}$,
for some $f_i\in\New_{M_i}$ with
 $M_i|N$ and the abelian varieties $A_{f_i}$ are pairwise non-isogenous over $\Q$. Any $f_i$ determines $n_i$ normalized eigenforms $g_j$ in
 $S_2(N)^{W_N}$ such that
  $J_0(N)^{W_N}\stackrel{\Q}\sim \prod_{j=1}^{r} A_{g_j}$, where $r=\sum_{i=1}^s n_i$ and $g_1,\cdots, g_{r}$ are all of these eigenforms.
 The basis of the Galois conjugates of the newforms $f_i$  together the exponents $n_i$ allow us to compute $|X_0/W_N(\F_{p^n})|$ for all primes
 $p\nmid N$,
 thanks to the Eichler-Shimura congruence.  The basis of the  regular differentials formed by all Galois conjugates of $g_j(q) \, dq/q$ allows us  to
 compute equations for $X_0/W_N$ by use of
 a theorem of Petri in the non-hyperelliptic case.
\end{rem}

Thus for quotient modular curve, we can carry out all such
computations and decide if is bielliptic or not. See examples in the
proofs of Lemmas \ref{lem7.5}, \ref{lem7.6}, \ref{lem6.4} and
\ref{6b.9} and github folder
\begin{verbatim}
https://github.com/FrancescBars/Mathematica-files-on-Quotient-Modular-Curves
\end{verbatim}
for different source and computations done by MATHEMATICA applying
the above results from Petri's theorem.

\section{Bielliptic quotients with $g_N^*=0$}
From the tables in the Appendix we see that for these levels most
quotients have genus $1$ or even $0$. So we only have to examine the
following $13$ curves.

\begin{itemize}

\item $X_0(44)/w_4 \cong X_0(22)$ by Proposition \ref{4.10}. Since $X_0(22)$ has the two bielliptic involutions $w_2$ and $w_{22}$ by \cite{Ba1}, from Proposition \ref{4.10} we also see that
$S_2 w_2^{(22)} S_2$ and $S_2 w_{22}^{(22)} S_2$ are bielliptic
involutions of $X_0(44)/w_4$.


\item $X_0(54)/w_2$ of genus $2$, has Jacobian decomposition $E27a\times E54a$, so all automorphisms are defined over $\QQ$. By MAGMA the automorphism group is $\mathbb{Z}/2\mathbb{Z}$,
 so the only involution is the hyperelliptic one and the curve is not bielliptic.

\item $X_0(56)/w_8$ is bielliptic. Namely, by Lemma
\ref{4lemS4.3} the involution $V_2 w_7$ has $8$ fixed points on
$X_0(56)$, so it has at least $4$ fixed points on $X_0(56)/w_8$.
Thus it is a bielliptic involution or the hyperelliptic involution.
But the hyperelliptic involution obviously is $w_7$.

\item $X_0(92)/w_4$ of genus $5$ and $X_0(92)/w_{92}$ of genus $4$ are both hyperelliptic, so by the Castelnuovo inequality they cannot be bielliptic.

\item $X_0(60)/w_{12}$, by the same argument, is hyperelliptic of genus $4$ and thus not bielliptic.

\item $X_0(60)/w_4 \cong X_0(30)$ and $X_0(60)/\langle w_4,
w_3\rangle\cong X_0(30)/w_3$ again by Proposition \ref{4.10}. As
before, the bielliptic involutions (compare \cite{Ba1} and
\cite{BaGonKa}) conjugate back to the bielliptic involutions $w_5$,
$S_2 w_6^{(30)} S_2$ and $S_2 w_{30}^{(30)} S_2$ of $X_0(60)/w_4$
resp. $S_2 w_2^{(30)} S_2$ and $S_2 w_{10}^{(30)} S_2$ of
$X_0(60)/\langle w_4, w_3\rangle$.

\item $X_0(60)/w_4$ and $X_0(60)/w_{60}$ each have a bielliptic Atkin-Lehner involution because they map of degree $2$ to the elliptic curve $X_0(60)/\langle w_4, w_5\rangle$ resp. $X_0(60)/\langle w_3, w_{20}\rangle$.(See also the previous item.)
 Noting that
$X_0(60)/\langle w_3, w_5\rangle$ also has genus $1$ we see that
each of $X_0(60)/w_3$, $X_0(60)/w_5$ and $X_0(60)/w_{20}$ maps to
two of these three elliptic curves and hence has two bielliptic
AL-involutions.

\item $X_0(60)/\langle w_5, w_{12}\rangle$ has genus $2$. Its Jacobian decomposition over $\QQ$ is $E20a\times E30a$, therefore all automorphisms are defined over $\QQ$.
By MAGMA its automorphism group over $\QQ$ is
$\mathbb{Z}/2\mathbb{Z}$, so there is no bielliptic involution.

\end{itemize}

\section{Bielliptic quotient curves when $X_0^*(N)$ has
genus $1$}

Part {{(i)}} of Theorem \ref{mainpq} is almost self-evident, as
these are exactly the values of $N$ and $W_N$ for which $X_0^*(N)$
has genus $1$ and there is a degree $2$ map from $X_0(N)/W_N$ to it.
From the tables in the Appendix we see that with the exception of
the elliptic curves $X_0^+(40)$, $X_0^+(48)$, $X_0^+(63)$ and
$X_0^+(75)$ these curves $X_0(N)/W_N$ do have genus at least $2$.

In the remainder of this section we finish the case $g_N^* =1$ by
deciding the curves $X_0(N)/w_d$ where $X_0^*(N)$ is elliptic and
$N$ has $3$ different prime divisors.

\begin{lema}\label{lemma7.1} The following $48$ quotient curves $X_0(N)/w_d$ are not
bielliptic

\begin{center}
\begin{tabular}{|l|c|}
\hline $N$&$(N,w_d)$\\
\hline 84&$(84,w_7);(84,w_{28});(84,w_{21})$\\
\hline 90&$(90,w_2);(90,w_{10});(90,w_{18})$\\
\hline 120&$(120,w_5);(120,w_8);(120,w_3),(120,w_{40})$,\\
\hline 126&$(126,w_2);(126,w_7);(126,w_{18})$;\\
\hline 132&$(132,w_d),d||132$\\
\hline 140&$(140,w_d),d||140$\\
\hline 150&$(150,w_d),d||150$\\
\hline 156&$(156,w_d),d||156$\\
\hline 220&$(220,w_d),d||220$\\

\hline
\end{tabular}
\end{center}

\end{lema}

\begin{proof} From the tables in the Appendix we see that for each such
curve $X_0(N)/w_d$ there exists a suitable Atkin-Lehner involution
$w_m$ such that $g(X_0(N)/w_d)\geq 2 g(X_0(N)/\langle w_d,
w_m\rangle)$ and $X_0(N)/\langle w_d, w_m\rangle$ is not
subhyperelliptic. So $X_0(N)/w_d$ cannot be bielliptic by Lemma
\ref{unramifiedcovering}.
\end{proof}

\begin{rem}\label{rem7.2} There are several other methods by which one could prove
a large subset of the $48$ curves in the previous lemma to be not
bielliptic.
\par
Note that all these curves have a map of degree $4$ to the genus $1$
curve $X_0^*(N)$. So if $g(X_0(N)/w_d)> 10$ a hypothetical
bielliptic map would by the Castelnuovo inequality have to factor
over a common quotient curve with this degree $4$ map, i.e. the
bielliptic involution would have to be an Atkin-Lehner involution,
which by the tables does not exist.
\par
Alternatively, again because $\Aut (X_0(N)/w_d)$ has a subgroup of
order $4$, Proposition \ref{prop3.1} shows that those with even
genus $g\geq 6$ are not bielliptic.
\par
Finally, one could also use Lemma \ref{lemadisca} with $p=3$ to
exclude all curves with $N=220$.
\end{rem}

\begin{lema}\label{lem7.3} The curve $X_0(126)/w_9$ is bielliptic with bielliptic
involution $V_3 w_7$. The genus $5$ curve $X_0(126)/w_{63}$ has at
least two bielliptic involutions, namely $V_3$ and $V_3 w_9$, both
defined over $\QQ(\sqrt{-3})$. The two genus $7$ curves
$X_0(126)/w_{14}$ and $X_0^+(126)$ are isomorphic and not
bielliptic.
\end{lema}
\begin{proof}
From Lemma \ref{4.19} we see that $V_3$ has the same number of fixed
points as $w_9$, namely none. And $V_3 w_9$, being a Galois
conjugate of $V_3$, also has the same number of fixed points. See
Lemma \ref{lemmaV3}.

By exactly the same arguments each of the involutions $w_{63}$, $V_3
w_7$ and $V_3 w_{63}$ has $16$ fixed points. Now with Lemma
\ref{4.13} one easily checks that the modular curves
$X_0(126)/\langle w_9, V_3 w_7\rangle$, $X_0(126)/\langle w_{63},
V_3\rangle$ and $X_0(126)/\langle w_{63}, V_3 w_7\rangle$ have genus
$1$.
\par
By Proposition \ref{4.12} the curves $X_0(126)/w_{14}$ and
$X_0^+(126)$ are isomorphic, and by \cite{Jeon} the latter one is
not bielliptic.
\end{proof}

\begin{lema}\label{lem7.4} The involution $V_2 w_{40}$ induces on each of the curves
$X_0(120)/w_{15}$, $X_0(120)/w_{24}$ and $X_0^+(120)$ a bielliptic
involution. Moreover, $X_0(120)/w_{15}$ has exactly two more
bielliptic involutions, namely $S_2$ and $w_8S_2w_8$.
\end{lema}
\begin{proof}
From Lemma \ref{4lemS4.3} we obtain the following table:
\begin{center}
\begin{tabular}{|c|c||c|c|}
\hline
$v$ & $\#(v,X_0(120))$ & $v$ & $\#(v,X_0(120))$\\
\hline
 $id$&      -         &            $V_2$ &0\\
$w_8$&0&$V_2 w_8$& 0 \\
$w_3$& 0&$V_2 w_3$& 8 \\
$w_5$& 0& $ V_2 w_5$&0\\
 $w_{24}$& 8&$ V_2 w_{24}$& 0 \\
 $w_{40}$ & 0&$ V_2 w_{40}$& 16\\
 $ w_{15}$ &16&$ V_2 w_{15}$& 8\\
  $w_{120}$& 8&$ V_2 w_{120}$& 0\\
\hline
\end{tabular}
\end{center}

Now the genus of $X_0(120)/\langle w_d, V_2 w_{40}\rangle$ can be
easily calculated using Lemma \ref{4.13}.

Since $X_0(120)/w_{15}$ maps with degree $2$ to the elliptic curve
$X_0(60)/w_{15}$ isomorphic to $X_0(120)/\langle w_8 S_2 w_8,
w_{15}\rangle$ (compare Lemma \ref{lem2S4.3}), $w_8 S_2 w_8$ is also
a bielliptic involution of $X_0(120)/w_{15}$. Furthermore, $S_2$ is
conjugate to $w_8 S_2 w_8$ in the automorphism group of
$X_0(120)/w_{15}$, and hence also a bielliptic involution. Finally,
since $S_2(w_8 S_2 w_8)V_2 w_{40}=w_5$ and $X_0(120)/\langle w_{15},
w_5\rangle$ is non-hyperelliptic of genus $3$, by \cite[Remark
3.2]{KMV} there are no further bielliptic involutions of
$X_0(120)/w_{15}$.

\end{proof}

\begin{lema}\label{lem7.5} $V_3 w_{10}$ is a bielliptic involution of $X_0(90)/w_9$.
And $X_0^+(90)$ actually has (at least) two, namely $V_3$ and $V_3
w_9$. The curves $X_0(90)/w_5$ and $X_0(90)/w_{45}$ on the other
hand are not bielliptic. \end{lema}

\begin{proof}
As before we get from Lemma \ref{4.19} that each of $w_9$, $V_3$ and
$V_3 w_9$ has $4$ fixed points, and each of $w_{90}$, $V_3 w_{10}$,
$V_3 w_{90}$ has $8$. Then we use Lemma \ref{4.13} to check the
genus of the quotient curves in question.


We have
$$J_0(90)^{\langle w_5\rangle }\sim_{\mathbb{Q}}(E30a)^2\times(E45a)^2\times
E90b$$
$$J_0(90)^{\langle w_{45}\rangle }\sim_{\mathbb{Q}} (E15a)^2\times E30a\times
E90b\times E90c$$

For $J_0(90)^{\langle w_{45}\rangle}$ there is a quadratic twist
$E30a\sim_{\Q(\sqrt{-3})}E90c$. Over the rationals we have for
$p=11$ and $E=E30a$ or $E90c$ does not satisfy (\ref{eq4.1}) because
$|\#X_0(90)/w_{45}(\mathbb{F}_{11})-2*\#E(\mathbb{F}_{11})|=2$, and
$\dim \mathcal{L}_{2,E90b}<\dim \mathcal{L}_2$. So remains if $E15a$
is or not a bielliptic quotient over $\Q$, but is not possible
because there does not exist any matrix
$\mathcal{A}=\left(\begin{array}{cc} a_1&b_1\\
a_2&b_2\\ \end{array}\right)\in\rm{GL}_2(\mathbb{Q})$ satisfying the
condition:
\begin{equation}\label{x1x2g5}
Q_2(a_1 x_1+a_2 x_2,b_1 x_1+b_2 x_2,x_3,x_4,x_5)=Q_2(-a_1 x_1+a_2
x_2,-b_1 x_1+b_2 x_2,x_3,x_4,x_5)
\end{equation}
 for all $Q_2\in\mathcal{L}_2$.

 Now for $\mathbb{Q}(\sqrt{-3})$ (by the quadratic twist) we
 have the Jacobian decomposition
 $$J_0(90)^{\langle w_{45}\rangle}\sim_{\mathbb{Q}(\sqrt{-3})} (E15a)^2\times
 (E30a)^2\times E90b$$
 but there does not exist any matrix
$\mathcal{A}=\left(\begin{array}{cc} a_1&b_1\\
a_2&b_2\\ \end{array}\right)\in\rm{GL}_2(\mathbb{Q}(\sqrt{-3}))$
satisfying the condition (\ref{x1x2g5}) and
\begin{equation}\label{x3x4g5}Q_2(x_1,x_2,a_1 x_3+a_2 x_4,b_1 x_3+b_2
x_4,x_5)=Q_2(x_1,x_2,-a_1 x_3+a_2 x_4,-b_1 x_3+b_2 x_4,x_5)
\end{equation} for all $Q_2\in\mathcal{L}_2$.

For $J_0(90)^{\langle w_5\rangle}$ we obtain that there are no
quadratic twists, thus any automorphism of the curve is defined over
$\mathbb{Q}$. Lemma \ref{lemadegree} discards $E90b$ as elliptic
quotient. Applying Proposition \ref{exponent} we obtain that there
does not exist bielliptic quotient. Similarly because there is no
matrix $\mathcal{A}\in \rm{GL}_2(\Q)$ as above that satisfies
(\ref{x1x2g5}), thus $E30a$ is not a bielliptic quotient, and
similarly we discard $E45a$ because no matrix $\mathcal{A}$ as above
that satisfies equation (\ref{x3x4g5}).

After we observe that $J_0(90)^{\langle
w_{45}\rangle}\sim_{\mathbb{Q}(\sqrt{-3})} J_0(90)^{\langle
w_5\rangle}$ because $E15a\sim_{\mathbb{Q}(\sqrt{-3})}E45a$. Recall
that $X_0(90)/w_5$ and $X_0(90)/w_{45}$ are isomorphic by use of
$V_3$.

See all computation details in name files related to above quotient
modular curves in the folder
\begin{verbatim}
https://github.com/FrancescBars/Mathematica-files-on-Quotient-Modular-Curves
\end{verbatim}.

\end{proof}

\begin{lema}\label{lem7.6} The curve $X_0(84)/w_4$ has $S_2 w_{14}^{(42)} S_2$ as a bielliptic involution. $X_0^+(84)$ is also bielliptic. But $X_0(84)/w_3$ and $X_0(84)/w_{12}$ are not bielliptic.
\end{lema}

\begin{proof}This follows from Proposition \ref{4.10} and \cite{Ba1} resp. from \cite{Jeon}.
\par
For the other two curves we use Proposition \ref{exponent}. The
Jacobian decomposition over $\Q$ is

$$J_0(84)^{\langle w_3\rangle} \sim (E14a)^2\times(E42a)^2\times E84b$$
$$J_0(84)^{\langle w_{12}\rangle} \sim (E14a)^2\times (E21a)\times(E42a)\times(E84a)$$

From the Jacobian decomposition of $J_0(84)^{\langle w_3\rangle}$
and $J_0(84)^{\langle w_{12}\rangle}$ all endomorphisms are defined
over $\mathbb{Q}$ (no quadratic twist in the elliptic curves
involved and is the same decomposition in the algebraic closure of
the rationals), see Proposition \ref{defi}. The factors with power 1
are discarded because $\dim \mathcal{L}_{2,i}<\dim \mathcal{L}$,
thus are not bielliptic quotients. The bielliptic quotient $E14a$ is
not possible because there does not exist any matrix
$\mathcal{A}=\left(\begin{array}{cc} a_1&b_1\\
a_2&b_2\\ \end{array}\right)\in\rm{GL}_2(\mathbb{Q})$ satisfying
(\ref{x1x2g5})for all $Q_2\in\mathcal{L}_2$. Similarly there does
not exist such a matrix for $E42a$ for $J_0(84)^{\langle
w_3\rangle}$ satisfying (\ref{x3x4g5}) for all
$Q_2\in\mathcal{L}_2$. Thus by Proposition \ref{exponent}, they are
not bielliptic.

See all computation details in \begin{verbatim}
https://github.com/FrancescBars/Mathematica-files-on-Quotient-Modular-Curves
\end{verbatim}.

\end{proof}

\section{Quotient modular curves of level $N$ with $X_0^*(N)$ hyperelliptic}

\subsection{Quotient modular curves with $X_0^*(N)$ hyperelliptic and
$N$ having two prime divisors}

In this subsection we discuss the candidates for which $X_0^*(N)$ is
hyperelliptic and $N$ is only divisible by two different primes. So
we treat the $15$ values $$N=88, 104, 112, 116, 117, 135, 147, 153,
184, 284, 136, 171, 207, 176, 279.$$

We emphasize that the curves $X_0^+(N)$ have already been treated in
\cite{Jeon} and are not listed in our Theorem \ref{mainpq}. So we
largely ignore them and only mention a few bielliptic ones among
them in passing.

\begin{lema}\label{lem6.0} The curves $X_0(284)/w_d$ are not
bielliptic.
\end{lema}
\begin{proof} This follows easily from Lemma \ref{lemadisca} by
counting $\FF_9$-rational points.
\end{proof}

\begin{lema}\label{lem6.1} The curves $X_0(184)/w_8$,
$X_0(207)/w_9$, $X(279)/w_9$ and $X_0(279)/w_{31}$ are not
bielliptic.
\end{lema}

\begin{proof} These curves have an involution whose quotient is the
hyperelliptic curve $X_0^*(N)$. By the Hurwitz formula the number of
fixed points is $2g(X_0(N)/w_d)-4g(X_0^*(N))+2>8$ (see the tables in
the Appendix). So by Lemma \ref{4.1} they are not bielliptic.
\end{proof}

\begin{lema}\label{rem8.3}
We quickly decide some more curves.
\par
$X_0(88)/w_{11}$, $X_0(112)/w_7$ and $X_0(184)/w_{23}$ are
bielliptic, because they map of degree $2$ to the elliptic curves
$X_0(44)/w_{11}$, $X_0(56)/w_7$, $X_0(92)/w_{23}$ (see the table for
$g_N^* =0$ in the Appendix). So by Lemma \ref{lem2S4.3} a bielliptic
involution is given by $w_8 S_2 w_8$ (resp. $w_{16}S_2 w_{16}$ for
$X_0(112)/w_7$). Moreover, $S_2$ is a conjugate bielliptic
involution.
\par
On the other hand, $X_0(116)/w_4$ is by Proposition \ref{4.10}
isomorphic to $X_0(58)$, and hence not bielliptic by \cite{Ba1}.
\par
Also, $X_0(153)/w_{17}$ and $X_0(207)/w_{23}$ are by Proposition
\ref{4.12} isomorphic to $X_0^+(153)$ resp. $X_0^+(207)$ and so not
bielliptic by \cite{Jeon}.
\end{lema}

\begin{lema}\label{lem6.2} Let $N\in\{104, 117, 136, 171, 176\}$ and $i=2$ resp. $3$
depending on whether $N$ is even or odd. Then $V_i w_N$ is a
bielliptic involution of $X_0(N)/\varpi_1$ and of $X_0^+(N)$,
whereas $X_0(N)/\varpi_2$ is not bielliptic. Here we are using the
notation $\varpi_1$ and $\varpi_2$ from Section \ref{2} (xii), which
allows us to simultaneously describe the Atkin-Lehner involutions
for different $N$.

\end{lema}

\begin{proof} Thanks to Lemmas \ref{4lemS4.3} respectively
\ref{4.19} we can calculate the number of fixed points of every
involution $V_i w_d$. We only carry this out for $N=176$, for which
we obtain the following table
\begin{center}
\begin{tabular}{|c|c||c|c|}
\hline $v$ & $\#(v,X_0(176))$ & $v$ & $\#(v,X_0(176))$ \\
\hline id&                - &                $V_2$& 0\\
$ w_{16}$& 0& $V_2 w_{16}$&        4\\
$ w_{11}$ &      0&$ V_2 w_{11}$ &12\\
 $w_{176}$  &  12&  $V_2
w_{176}$& 24\\
\hline
\end{tabular}
\end{center}

With Lemma \ref{4.13} we can thus check that $X_0(176)/w_{11}$ has
no bielliptic involution in $\langle B(N), V_2\rangle$ whereas $V_2
w_{176}$ is a bielliptic involution for the other two curves.
\par
To finish the proof we note that by Lemma \ref{4.2} a bielliptic
involution $v$ on $X_0(176)/w_{11}$ would induce the hyperelliptic
involution on $X_0^*(176)$, which according to the table above is
$V_2$. So $v$ could only be $V_2$ or $V_2 w_{16}$.

\end{proof}

\begin{lema}\label{lem6.3} The curve $X_0(112)/w_{16}$ is not
bielliptic.
\end{lema}

\begin{proof} By exactly the same arguments as in the previous lemma we can
show that $\langle B(112), V_2\rangle$ contains no bielliptic
involution for this curve. However, now the problem is that we don't
know the hyperelliptic involution of $X_0^*(112)$. (Actually, $V_2$
does not help because it is a bielliptic involution of
$X_0^*(112)$.) But Proposition \ref{prop3.1} guarantees that the
bielliptic involution of $X_0(112)/w_{16}$ would have to be
contained in $\langle w_7, V_2\rangle$.
\end{proof}

\begin{lema}\label{lem6.3b} The genus $7$ curve $X_0(116)/w_{29}$ is not bielliptic.
\end{lema}

\begin{proof} The involutions $w_4$ and $S_2$, both defined over $\QQ$,
generate a group $H\cong S_3$ of automorphisms of
$X=X_0(116)/w_{29}$. Each of the three (conjugate) involutions has
$8$ fixed points on $X$.
\par
If $v$ is a bielliptic involution of $X$, then $H$ induces an
isomorphic group $\widetilde{H}$ of automorphisms on $X/v$. Under
the action of $v$ on the $8$ fixed points of an involution in $H$
these can at worst fall together in pairs. So each involution in
$\widetilde{H}$ has at least $4$ fixed points on $X/v$. By the
Hurwitz formula the curve $(X/v)/\widetilde{H}$ thus has genus $0$
and the automorphisms of order $3$ in $\widetilde{H}$ have no fixed
points on $X/v$. So the genus $1$ curve $X/v$ has a fixed point free
automorphism of order $3$, defined over $\QQ$, and hence it must
have a $\Q$-rational $3$-torsion point. But the only elliptic curves
over $\Q$ of level $M$ properly dividing $116$ have $M=58$ and no
$3$-torsion by \cite{Cre}. And for $M=116$ the modular degrees of
the optimal elliptic curves $116a1$, $116b1$, $116c1$ are $120$, $8$
and $15$, and hence too big by Lemma \ref{lemadegree}.
\end{proof}

\begin{lema}\label{lem6.4} The remaining six curves, that is the quotient curves $X_0(88)/w_8$, $X_0(135)/w_{27}$, $X_0(135)/w_5$, $X_0(147)/w_3$,
$X_0(147)/w_{49}$, $X_0(153)/w_9$, are not bielliptic.
\end{lema}

\begin{proof}
Ordering by genus, the Jacobian decompositions over $\Q$ are as
follows:
$$\begin{array}{ccl}
J_0(88)^{\langle w_8\rangle}&\sim& (E11a)^2\times E44a\times E88a\\
J_0(147)^{\langle w_3\rangle}&\sim& E49a\times E147a\times
E147c\times A_f; \dim
A_f=2,\; f\in \New(147)\\
\hline J_0(135)^{\langle w_5\rangle}&\sim& E27a\times (E45a)^2\times
E135a\times
A_f;\dim A_f=2, f\in \New(135)\\
J_0(147)^{\langle w_{49}\rangle}&\sim&E21a\times E147b\times
A_f\times A_g;
\dim(A_f)=\dim(A_g)=2;\; f,g\in \New(147)\\
J_0(135)^{\langle w_{27}\rangle}&\sim&(E15a)^2\times E45a\times
E135a\times E135b\times
A_f;\; \dim A_f=2; \; f\in \New(135)\\
J_0(153)^{\langle w_9\rangle}&\sim&(E17a)^2\times E51a\times
A_f\times E153a\times
E153b;\; \dim A_f=2;\ f\in \New(51)\\
\end{array}
$$

For the genus $4$ cases there are no quadratic twists, thus all
endomorphisms are defined over $\mathbb{Q}$ by Proposition
\ref{defi}. By Remark \ref{2.7} we consider $\mathcal{L}_2$ and
$\mathcal{L}_3$, see further details in \cite[\S6]{BaGon2}. We
compute a nonzero polynomial $Q_2\in\mathcal{L}_2$ which generates
$\mathcal{L}_2$: $48t^2+x^2+xy-7y^2-16z^2$. For $E44a$ and $E88a$
corresponding to $z,t$ we have that
$Q_2(\omega_1,\omega_2,\omega_3,\omega_4)=Q_2(\omega_1,\omega_2,-\omega_3,\omega_4)$
and similar for $\omega_4$, but
\begin{equation}\label{equationq3}
\frac{Q_3(\ldots,\omega_i,\ldots)-Q_3(\ldots,-\omega_i,\ldots)}{x_i}\notin\mathcal{L}_2
\end{equation}
where $Q_3\in\mathcal{L}_3$ that is not a multiple of
$\mathcal{L}_2$, thus not bielliptic. For $E11a$, we make a change
of variables $w_1,w_2$ taking $Q_2(x,y,z,t)=48 t^2 - 176 x^2 + y^2 -
16 z^2$ but this does not satisfy the condition on $\mathcal{L}_3$
of equation (\ref{equationq3}) thus is not bielliptic. See the
computational details of this example and the remaining ones in
files related with such quotient modular curves in the folder
\begin{verbatim}
https://github.com/FrancescBars/Mathematica-files-on-Quotient-Modular-Curves
\end{verbatim}
We only recall that the bielliptic involution is defined over
$\mathbb{Q}$ for genus $\geq 6$ (in such cases it is unique) thus we
only need to worry if a bielliptic involution exists and is not
defined over $\Q$ for the genus 5 curve $X_0(147)/w_3$. The
dimension two factor does not have any elliptic quotient by use of
Proposition \ref{Pyle}, and there only appears an inert twist
$E49a\sim_{\Q(\sqrt{-7})} E49a$, therefore the bielliptic
involution, if it exists, is defined over $\Q(\sqrt{-7})$. But
because over $\Q(\sqrt{-7})$ we have the same Jacobian
decomposition, we conclude as we did over $\Q$ that there is no
bielliptic involution.

Thus for all remaining situations, any dimension one factor with
power 1 in the Jacobian we only need to observe that $\dim
\mathcal{L}_{2,E}<\dim \mathcal{L}_2$ and for higher power, a
similar argument used in equation (\ref{x1x2g5}) in the previous
section, with a correct choice of the variables, shows that no
bielliptic involution appears by Proposition \ref{exponent}.

\end{proof}

\subsection{Quotient modular curves with $X_0^*(N)$ hyperelliptic and
$N$ having $3$ prime divisors}

In this part we deal with the cases where $X_0^*(N)$ is
hyperelliptic and $N$ is divisible by $3$ different primes. So $N$
is one of

$168$, $180$, $198$, $204$, $276$, $380$ (all with $g_N^* =2$) or
$252$, $315$ (both with $g_N^* =3$).

\begin{prop}\label{6b.1} Let $N$ be one of the eight numbers just
listed. Then none of the curves $X_0(N)/w_d$ is bielliptic.
\end{prop}

\begin{proof} From \cite{Jeon} we know that none of the eight curves
$X_0^+(N)$ is bielliptic. So we can assume $d\neq N$. We consider
the covering $X_0(N)/w_d \to X_0(N)/\langle w_d, w_N\rangle$. Since
$w_N$ always has fixed points, by Lemma \ref{4.2} a bielliptic
involution of $X_0(N)/w_d$ would induce a hyperelliptic involution
on $X_0(N)/\langle w_d, w_N\rangle$. But by the tables in the
Appendix none of these curves is hyperelliptic.
\end{proof}

Now we investigate the curves $X_0(N)/W$ with $|W|=4$.

\begin{lema}\label{6b.2} $V_3 w_7$ is a bielliptic involution of the curves
$X_0(252)/\langle w_4, w_9\rangle$, $X_0(252)/\langle w_9,
w_7\rangle$ and $X_0(252)/\langle w_4, w_{63}\rangle$. Moreover
$V_3$ is a second bielliptic involution of $X_0(252)/\langle w_4,
w_{63}\rangle$.

By contrast, $X_0(252)/\langle w_4, w_7\rangle$, $X_0(252)/\langle
w_9, w_{28}\rangle$, $X_0(252)/\langle w_7, w_{36}\rangle$ and
$X_0(252)/\langle w_{36}, w_{28}\rangle$ are not bielliptic.
\end{lema}

\begin{proof} Using Lemma \ref{4.19} we obtain the
following table with number of fixed points.
\begin{center}
\begin{tabular}{|c|c||c|c|}
\hline $v$ & $\#(v,X_0(252))$ & $v$ & $\#(v,X_0(252))$ \\
\hline
 id&   -&          $ V_3$& 0\\
 $ w_4$& 8&                $V_3 w_4$& 0\\
 $ w_9$& 0& $V_3w_9$&             0\\
 $w_7$&              0 &        $V_3 w_7$& 24\\
$w_{36}$&         0&             $V_3 w_{36}$&        0\\
$ w_{28}$& 0& $V_3 w_{28}$&        8\\
$ w_{63}$&        24& $V_3 w_{63}$& 24\\
$ w_{252}$& 8& $V_3 w_{252}$& 8\\
\hline
\end{tabular}
\end{center}

With Lemma \ref{4.13} we can then easily check three things, namely
that $V_3 w_7$ is a bielliptic involution for the first three curves
(and $V_3$ as well for $X_0(252)/\langle w_4, w_{63}\rangle$), that
the other four curves have no bielliptic involution in $\langle
B(N), V_3\rangle$, and that $V_3$ induces the hyperelliptic
involution on $X_0^*(N)$.
\par
Combining the second and third fact proves that the four curves are
not bielliptic, because a bielliptic involution would by Lemma
\ref{4.2} induce the hyperelliptic involution on $X_0^*(N)$, and
hence would be in $\langle B(N), V_3\rangle$.
\end{proof}

\begin{lema}\label{6b.3} $V_2 w_{168}$ is a bielliptic involution of the curves
$X_0(168)/\langle w_8, w_3\rangle$, $X_0(168)/\langle w_8,
w_7\rangle$, $X_0(168)/\langle w_3, w_{56}\rangle$ and
$X_0(168)/\langle w_7, w_{24}\rangle$.

Moreover, $V_2 w_8$ is a second bielliptic involution of
$X_0(168)/\langle w_3, w_{56}\rangle$.

The remaining curves $X_0(168)/\langle w_3, w_7\rangle$,
$X_0(168)/\langle w_8, w_{21}\rangle$ and $X_0(168)/\langle w_{24},
w_{56}\rangle$ are not bielliptic.
\end{lema}

\begin{proof} Apart from using the involutions $V_2 w_d$ the proof is
completely analogous to that of Lemma \ref{6b.2}, except for the
case of $X_0(168)/\langle w_{24},w_{56}\rangle$, in which case we
cannot apply Lemma \ref{4.2} because the genus is $4$. But being
hyperelliptic this curve therefore cannot also be bielliptic by the
Castelnuovo inequality (Lemma \ref{4.4}).

\end{proof}

Unfortunately the method we just used does not really work for the
six other values of $N$, as for them we don't know the hyperelliptic
involution of $X_0^*(N)$. For example, $V_3$ induces a bielliptic
involution on $X_0^*(180)$, $X_0^*(198)$ and $X_0^*(315)$.

\begin{lema}\label{6b.4} None of the curves $X_0(380)/W$ with $|W|=4$ is
bielliptic.
\end{lema}

\begin{proof} Assume that such a curve is bielliptic and let $E$ be the
elliptic curve it covers. Then $\overline{E}$ (the reduction of $E$
modulo $3$) would by Lemma \ref{lemadisca} have the maximally
possible number of $16$ $\FF_9$-rational points. So $\overline{E}$
must necessarily be supersingular. This is equivalent to
$j(\overline{E})=0$. This in turn is equivalent to the condition
that for the coefficients $a_i$ of the global minimal model of $E$
the expression $a_1^2 +a_2$ is divisible by $3$. With this one can
quickly exclude almost all candidates for $E$ in \cite{Cre}. Only
the isogeny classes $190a$ and $380a$ remain. But the curves in
$190a$ have $7$ rational points over $\FF_3$ and hence the same
number over $\FF_9$. And for $380a$ the degree of the strong Weil
uniformization is $24$ for the elliptic curve $380a1$ and $240$ for
$380b1$, so both too big by Lemma \ref{lemadegree}.
\end{proof}

In the previous lemma the four curves of genus bigger than $10$
could also have been quickly excluded using the following method.

\begin{lema}\label{6b.5} The curves $X_0(276)/W$ with $W$ any one of $\langle w_4,
w_3\rangle$, $\langle w_4, w_{69}\rangle$, $\langle w_3,
w_{92}\rangle$, $\langle w_{12}, w_{92}\rangle$, as well as
$X_0(204)/\langle w_4, w_3\rangle$ and $X_0(315)/\langle w_9,
w_7\rangle$ are all not bielliptic.
\end{lema}

\begin{proof} For all these curves the map to the (hyperelliptic) curve
$X_0^*(N)$ is given by an involution with more than $8$ fixed
points. So by Lemma \ref{4.1} they are not bielliptic.
\end{proof}

\begin{lema}\label{6b.6} None of the curves $X_0(315)/W$ with $|W|=4$ is
bielliptic.
\end{lema}

\begin{proof}
The curve $X_0(315)/\langle w_9, w_7\rangle$ has already been
excluded in the last lemma. So let $X$ be one of the other six
curves and assume it is bielliptic with corresponding elliptic curve
$E$. A necessary condition for this is that for the reductions
modulo $2$ we have
$$|\overline{X}(\FF_{2^k})|\leq 2|\overline{E}(\FF_{2^k})|$$
for all $k$. With MAGMA one can calculate
$|\overline{X}(\FF_{2^k})|$ for small values of $k$. It turns out
that $|\overline{X}(\FF_4)|=18$ for all six curves. Hence
$\overline{E}$ must have the maximally possible number of $9$
rational points over $\FF_4$. So it also only has $9$ rational
points over $\FF_{16}$. This excludes the curves $X_0(315)/\langle
w_9, w_5 \rangle$ and $X_0(315)/\langle w_9, w_{35} \rangle$ because
they both have $|\overline{X}(\FF_{16})|=26$.
\par
Moreover, $|\overline{E}(\FF_4)|=9$ implies $j(\overline{E})=0$.
This is equivalent to the coefficient $a_1$ of the global minimal
model of $E$ being even. So with one glance at \cite{Cre} one can
exclude that $E$ has conductor $15$, $21$, $45$, $63$ or $105$. Only
conductor $35$ or $315$ is possible. But for $315$ the degree of the
strong Weil uniformization is too big. And an elliptic curve with
conductor $35$ does not appear in the Jacobian of $X_0(315)/\langle
w_7, w_{45}\rangle$ and $X_0(315)/\langle w_{45}, w_{63}\rangle$.

$$J(X_0(315)/\langle w_7, w_{45}\rangle)\sim_{\mathbb{Q}} (E15a)\times
(E21a)^2\times A_{35,x^2+z-4}\times(E63a)\times A_{315,x^2+2x-1}$$
$$J(X_0(315)/\langle w_{45}, w_{63}\rangle)\sim_{\mathbb{Q}}
(E15a)\times(E21a)\times (E105a)\times(E315a)\times
A_{315,x^2+2x-1}\times A_{315,x^2-5}.$$

The other two curves of genus $8$ are isomorphic to the above two by
Lemma \ref{4.12}, so also not bielliptic.

\end{proof}

From Proposition \ref{4.10} we obtain the following isomorphisms
over $\QQ$.

\begin{lema}\label{6b.7}

$X_0(180)/\langle w_4, w_9\rangle\cong X_0(90)/w_9$, so by Lemma
\ref{lem7.5} and Proposition \ref{4.10} it has $S_2 V_3
w_{10}^{(90)} S_2$ as a bielliptic involution.

The following curves are not bielliptic by Section 7 resp.
\cite{BaGon2} resp. \cite{BaGonKa}.

$X_0(180)/\langle w_4, w_5\rangle \cong X_0(90)/w_5$;

$X_0(180)/\langle w_4, w_{45}\rangle \cong X_0(90)/w_{45}$;

$X_0(198)/\langle w_9, w_{11}\rangle \cong X_0^*(396)$;

$X_0(204)/\langle w_4, w_{17}\rangle \cong X_0(102)/w_{17}$;

$X_0(204)/\langle w_4, w_{51}\rangle \cong X_0(102)/w_{51}$;

$X_0(276)/\langle w_4, w_{23}\rangle \cong X_0(138)/w_{23}$.

\end{lema}

\begin{lema}\label{6b.8} The curve $X_0(204)/\langle w_3,  w_{68}\rangle$ is not
bielliptic. \end{lema}

\begin{proof} By \cite{HaSh06} this curve is trigonal. If it also were
bielliptic, then by the Castelnuovo inequality its genus could be at
most $4$. But it has genus $5$.
\end{proof}

\begin{lema}\label{6b.9} The remaining next $15$ curves are not
bielliptic.

\begin{center}
\begin{tabular}{|c|c|c|}
\hline $X_0(180)/\langle w_9, w_5\rangle$, & $X_0(180)/\langle w_9,
w_{20}\rangle$,&

$X_0(180)/\langle w_5, w_{36}\rangle$,\\

$X_0(180)/\langle w_{36}, w_{20}\rangle$,&

$X_0(198)/\langle w_2, w_9\rangle$,&

$X_0(198)/\langle w_2, w_{11}\rangle$,\\

$X_0(198)/\langle w_2, w_{99}\rangle$,&

$X_0(198)/\langle w_9, w_{22}\rangle$,&

$X_0(198)/\langle w_{11}, w_{18}\rangle$,\\

$X_0(198)/\langle w_{18}, w_{22}\rangle$,&

$X_0(204)/\langle w_3, w_{17}\rangle$,&

$X_0(204)/\langle w_{17}, w_{12}\rangle$,\\

$X_0(204)/\langle w_{12}, w_{51}\rangle$,&

$X_0(276)/\langle w_3, w_{23}\rangle$,&

$X_0(276)/\langle w_{23}, w_{12}\rangle$\\
\hline
\end{tabular}
\end{center}

\end{lema}

\begin{proof} First we consider the genus 5 curves. For that we need to
study quadratic twists because the possible bielliptic involution
(or automorphisms) could be not defined over $\Q$.

\begin{center}
\begin{tabular}{c}
$J_0(180)^{\langle w_9,w_{20}\rangle} \sim_{\Q} E15a\times
E30a\times
E36a\times E90a\times E90b$\\
$J_0(180)^{\langle w_5,w_{36}\rangle} \sim_{\Q} E20a\times
E45a\times (E30a)^2\times E90b$\\
$J_0(180)^{\langle w_{36},w_{20}\rangle}\sim_{\Q} E15a\times
E30a\times E90b\times
E90c\times E180a$\\
\hline $J_0(198)^{\langle w_2,w_{99}\rangle}\sim_{\Q} E11a\times
E33a\times E66a\times E99a\times E198a$\\
$J_0(198)^{\langle w_{11},w_{18}\rangle}\sim_{\Q} E66a\times
E66b\times E99a\times E99b\times E99d$\\
\hline $J_0(204)^{\langle w_{12},w_{51}\rangle}\sim_{\Q} E17a\times
E34a\times E102a\times E102b\times E204b$\\
\hline $J_0(276)^{\langle w_{23},w_{12}\rangle}\sim_{\Q} E69a\times
E92a\times E138a\times E138b\times E138c$
\end{tabular}
\end{center}
We find the following quadratic twists (we are only interested in a
fixed Jacobian, thus we do not compare here quadratic twist for
modular forms between different Jacobians):
$E90a\sim_{\Q(\sqrt{-3})}E90b$, $E30a\sim_{\Q(\sqrt{-3})}E90c$ and
$E36a\sim_{\Q(\sqrt{-3})}E36a$.

Anyway, we know by $V_3$ the following modular curves are isomorphic
over $\Q(\sqrt{-3})$:\newline $X_0(180)/{\langle
w_5,w_{36}\rangle}\cong X_0(180)/\langle w_{36},w_{20}\rangle$,
$X_0(198)/\langle w_2,w_{99}\rangle\cong X_0(198)/\langle
w_{11},w_{18}\rangle$ and \newline $X_0(198)/\langle
w_2,w_{11}\rangle\cong X_0(198)/\langle w_{18},w_{22}\rangle$. These
isomorphism induce isogenous Jacobians, by the following twists:
$E15a\sim_{\Q(\sqrt{-3})}E45a$, $E180a\sim_{\Q(\sqrt{-3})} E20a$,
$E11a\sim_{\Q(\sqrt{-3})} E99d$, $E33a\sim_{\Q(\sqrt{-3})} E99b$,
$E198a\sim_{\Q(\sqrt{-3})} E66b$, $E66a\sim_{\Q(\sqrt{-3})}E198b$,
$E66c\sim_{\Q(\sqrt{-3})}E198e$, and $E99a\sim_{\Q(\sqrt{-3})}E99c$.
Thus, over $\mathbb{Q}(\sqrt{-3})$:

\begin{center}
\begin{tabular}{c}
$J_0(180)^{\langle w_9,w_{20}\rangle} \sim_{\Q(\sqrt{-3})}
E15a\times E30a\times
E36a\times (E90b)^2$\\
 $J_0(180)^{\langle
w_{36},w_{20}\rangle}\sim_{\Q(\sqrt{-3})} E15a\times (E30a)^2\times
E90b\times E180a$\\
\end{tabular}
\end{center}
In the above situation $\dim\mathcal{L}_2=3$ and for each elliptic
curve with power one we obtain $\dim\mathcal{L}_{2,E}<3$, thus they
are not bielliptic quotients. For the $E^2$-factors with $E$ an
elliptic curve we prove that there is no $\mathcal{A}\in GL_2(K)$
where $K=\Q$ or $\Q(\sqrt{-3})$ respectively of the Jacobian
decomposition satisfying Proposition \ref{exponent}. Thus none of
these curves is bielliptic.

See the complete computation details on each curve and for the
following ones in this proposition in name files related with such
quotient modular curves in the folder
\begin{verbatim}
https://github.com/FrancescBars/Mathematica-files-on-Quotient-Modular-Curves
\end{verbatim}.

Now consider the genus $\geq 6$ curves. It is enough to study them
over the rationals (we are not interested where the general
endomorphism are defined for the problem on biellipticity). The
Jacobian decomposition is as follow:
\begin{center}
\begin{tabular}{c}
$J_0(180)^{\langle w_9,w_5\rangle}\sim_{\Q}
(E20a)^2\times(E30a)^2\times E36a\times(E90b)^2$\\
\hline $J_0(198)^{\langle w_2,w_{11}\rangle}\sim_{\Q} (E66a)^2\times
E99a\times E99b\times E99d\times E198e$\\
$J_0(198)^{\langle w_{18},w_{22}\rangle}\sim_{\Q} E11a\times
E33a\times E66a\times E66c\times E99c\times E198b$\\
$J_0(198)^{\langle w_2,w_9\rangle}\sim_{\Q}(E11a)^2\times E33a\times
E66a\times E99a\times E99c\times E198d$\\
$J_0(198)^{\langle w_9,w_{22}\rangle}\sim_{\Q}(E11a)^2\times
E33a\times E66a\times E66c\times E99a\times E99c$\\
 \hline $J_0(204)^{\langle
w_3,w_{17}\rangle}\sim_{\Q}(E34a)^2\times
A_{f,68}\times (E102a)^2\times E(204a)$\\
$J_0(204)^{\langle w_{17},w_{12}\rangle}\sim_{\Q} (E34a)^2\times
E51a\times A_{f_3,68}\times E102a\times E102c$\\
\hline $J_0(276)^{\langle w_3,w_{23}\rangle}\sim_{\Q}
E92a\times(E138a)^2\times(E138c)^2\times A_{f_4,276}$\\
\end{tabular}
\end{center}
where $\dim(A_{f,68})=\dim(A_{f_3,68})=\dim(A_{f_4,276})=2$ and the
number next to $f_i$ is the level where it appears as newform. By
Theorem \ref{Pyle} they do not give any elliptic quotient. Moreover,
for each non-repeated factor of dimension one in the Jacobian we
have $\dim(\mathcal{L}_{2,E})<\dim(\mathcal{L})$ and for the terms
$E^2$ with $E$ an elliptic curve, there is no matrix $\mathcal{A}\in
\rm{GL}_2(\Q)$ satisfying Proposition \ref{exponent}. All in all we
obtain that none is a bielliptic quotient.
\end{proof}

\section{Quotient modular curves of level $N$ with $X_0^*(N)$ not
subhyperelliptic}

\begin{prop} Let $N$ be non-square free such that $X_0^*(N)$ is not
subhyperelliptic. Then the only bielliptic curve $X_0(N)/W_N$ where
$W_N$ is a proper subgroup of $B(N)$ is $X_0(144)/\langle
w_{144}\rangle$.
\end{prop}

\begin{proof} Let $X_0(N)/W_N$ be bielliptic and $X_0^*(N)$ not
subhyperelliptic. Then $X_0^*(N)$ must be bielliptic by Lemma
\ref{lemab3.1}. So we are dealing with a subset of the levels in
case (v) of Theorem \ref{levelstudy2}. Furthermore, by Lemma
\ref{unramifiedcovering} (b) the covering $X_0(N)/W_N \to X_0^*(N)$
must be totally unramified.
\par
Next we point out that every Atkin-Lehner involution $w_d$ that has
a fixed point on $X_0(N)$ must be contained in $W_N$. Otherwise it
would induce an involution with fixed points on $X_0(N)/W_N$,
contradicting the fact that the map to $X_0^*(N)$ is totally
unramified. In particular, the full Atkin-Lehner involution $w_N$,
which by \cite[p.454]{Ogg} always has fixed points, must be
contained in $W_N$.
\\ \\
Now we are ready to settle $N=420=2^2\cdot 3\cdot 5\cdot 7$, the
only case in this paper where $N$ has more than by $3$ different
primes divisors. As just said, $w_{420}\in W_{420}$. And as $w_4$
fixes a cusp by \cite[Proposition 3]{Ogg}, also $w_4 \in W_{420}$.
Moreover, $w_{20}$, $w_{35}$, $w_{84}$ and $w_{140}$ also all have
fixed points by \cite[p.453]{Ogg}. So they also all must be
contained in $W_{420}$, leading to the contradiction
$W_{420}=B(420)$.
\\ \\
Next we treat the cases where $N$ is divisible by $3$ different
primes. Thus, $N$ is one of the following:

\begin{center}
\begin{tabular}{|l|c|}
\hline
 $N$ & $g_N^*$ \\ \hline \hline
$234$, $240$, $252$, $294$, $312$, $315$, $348$, $476$ & $3$ \\
$228$, $260$, $264$, $280$, $300$, $306$, $342$ & $4$ \\
$364$, $444$, $495$ & $5$ \\
$558$ & $7$ \\ \hline
\end{tabular}
\end{center}

Obviously it suffices to prove that there are no such bielliptic
curves with $|W_N|=4$. As necessarily $w_N\in W_N$, there are at
most three possible choices for $W_N$, namely $\langle\omega_1,
\omega_2 *\omega_3 \rangle$, $\langle\omega_2, \omega_1 *\omega_3
\rangle$ and $\langle\omega_3, \omega_1 *\omega_2 \rangle$.
Moreover, the covering $X_0(N)/W_N \to X_0^*(N)$ is unramified if
and only if $X_0(N)/W_N$ has genus $2g_N^* -1$. By the tables in the
Appendix the only curves surviving this test are $X_0(260)/\langle
w_4, w_{65}\rangle$ and $X_0(300)/\langle w_4, w_{75}\rangle$, both
of genus $7$. But they are also not bielliptic. By Proposition
\ref{4.10} they are isomorphic to $X_0(130)/\langle w_{65}\rangle$
and $X_0(150)/\langle w_{75}\rangle$, respectively. For the second
curve we have shown in Lemma \ref{lemma7.1} that it is not
bielliptic. The exact same proof works for $N=130$, or alternatively
(since $130$ is square-free) we could invoke \cite{BaGonKa}.
\\ \\
Finally, if $N$ has only two different prime divisors it belongs to
one of the following sets

$144$, $152$, $164$, $189$, $196$, $236$, $245$, $248$ (all with
$g_N^*=3$),

$148$, $160$, $172$, $200$, $224$, $225$, $242$, $275$ (all with
$g_N^*=4$).

As discussed earlier, then necessarily $W_N =\langle w_N \rangle$.
But by \cite[Theorem 1.1]{Jeon} $X_0^+(144)$ is the only such curve
that is bielliptic.

\end{proof}

\section{Quadratic points}

In this section we prove our second main result, Theorem
\ref{main2pq}.
\\ \\
Let $X$ be a curve of genus at least $2$ which has a $\QQ$-rational
point. Then by \cite[Theorem 2.14]{BaMom} $X$ has infinitely many
quadratic points if and only if $X$ is hyperelliptic or $X$ has an
involution $v$, defined over $\QQ$, such that $X/v$ is an elliptic
curve $E$ with positive rank over $\QQ$.
\par
Now we specialize to $X$ being of the form $X(N)/W_N$. The
hyperelliptic ones have already been determined in \cite{FM}. On the
other hand, if the bielliptic involution $v$ is defined over $\QQ$,
then the conductor of the elliptic curve $E$ will be a divisor $M$
of $N$. So our main tool will be \cite{Cre}.
\\ \\
We start with the cases where $g_N^* =1$ and $N$ has $3$ different
prime divisors. Among these, for $N=84$, $90$, $120$, $126$, $132$,
$140$, $150$ there are no divisors $M$ for which there exist
elliptic curves of positive rank. For $N=156$ and $220$ the only
possibility is $M=N$. So in that case the map from $X_0(N)$ to $E$
of degree $2|W_N|\leq 8$ must factor through the strong Weil
parametrization. But by \cite{Cre} for the elliptic curves of
positive rank the degree of that parametrization is $12$ for $N=156$
and $36$ for $N=220$.
\par
We also mention as a warning that if $M$ is a proper divisor of $N$
it seems that one cannot expect any help from the strong Weil
parametrization from $X_0(M)$ to $E$. For example, $X_0(22)$ has two
bielliptic involutions defined over $\QQ$, namely $w_2$ and
$w_{22}$. As there are no elliptic curves with conductor $22$, the
quotient curves must be isogenous to the elliptic curve $X_0(11)$.
But the bielliptic maps do not factor through the canonical map to
$X_0(11)$, which has degree $3$.
\\ \\
By the same token, for the cases where $g_N^* =1$ and $N$ has only
$2$ different prime divisors, we can exclude all levels $N$ except
$99$ and $124$.
\par
Indeed $Jac(X_0^*(99))\sim_{\QQ}E99a$, which has rank $1$,
furnishing the curves $X_0(99)/\langle w_d\rangle$ with infinitely
many quadratic points. For $N=124$ again the degree of the strong
Weil parametrization (here $6$) is too small.
\\ \\
By exactly the same method one can exclude all remaining bielliptic
curves in the table of Theorem \ref{mainpq} except for the following
two constellations: $N=171$ with $M=57$ and $N=176$ with $M=88$.
\par
So assume that there is a map of degree $2$, defined over $\QQ$,
from $X_0(171)/\langle w_9\rangle$ to the elliptic curve $E$ with
conductor $57$ and positive rank over $\QQ$. For the reduction
modulo $2$ one easily verifies $\#E(\FF_2)=5$. And with the general
formula
$$\#E(\FF_{q^2})=
\#E(\FF_q)(2q+2-\#E(\FF_q))$$ we see that $E$ does not acquire more
points over $\FF_4$. This leads to a contradiction in Lemma
\ref{lemadisca}.
\par
The exact same proof works for $N=176$ if we reduce the elliptic
curve with conductor $88$ modulo $3$ and count the $\FF_9$-rational
points.
\\ \\
Finally we provide some more information on the two curves that are
bielliptic but not over $\Q$.

\begin{lema}\label{overQsqrt-3}
The genus $5$ curve $X_0(126)/w_{63}$ has exactly two bielliptic
involutions, namely $V_3$ and $V_3 w_9$, both defined over
$\Q(\sqrt{-3})$. This curve has only finitely many points that are
quadratic over $\Q(\sqrt{-3})$.
\par
The same holds for the curve $X_0(252)/\langle w_4, w_{63}\rangle$,
which is isomorphic to $X_0(126)/w_{63}$ by Proposition \ref{4.10}.
Its two bielliptic involutions are $V_3$ and $V_3 w_7$, both defined
over $\Q(\sqrt{-3})$.
\end{lema}

\begin{proof}
In Lemma \ref{lem7.3} we exhibited the two bielliptic involutions of
$X_0(126)/w_{63}$. Since the curve $X_0(126)/\langle w_{63},
w_2\rangle$ has genus $2$, we can apply Lemma \ref{lemgenus2}. It
tells us that if $X_0(126)/w_{63}$ had further bielliptic
involutions, it would have one defined over $\Q$. So one checks with
Petri that $X_0(126)/w_{63}$ has no bielliptic involution over $\Q$.
More computationally, one could also show from the splitting of the
Jacobian that all automorphisms of $X_0(126)/w_{63}$ are defined
over $\Q(\sqrt{-3})$ and then use Petri to determine the bielliptic
involutions over that field.
\par
Finally one checks that the elliptic quotients of $X_0(126)/w_{63}$
by $V_3$ resp. $V_3 w_9$, namely the base change of $E14a$ to
$\Q(\sqrt{-3})$ has rank $0$ over $\Q(\sqrt{-3})$.
\end{proof}

\section*{Acknowledgements} We are very happy to thank Professor
Josep Gonz\'alez for his different comments and suggestions. We are
grateful to unknown referees for detailed and constructive reports
and for many suggestions that improved the presentation of the
paper.

\appendix

\section{Computations of the genus}

We list the genus of $X_0(N)/W_N$ with $W_N\leq B(N)$ such that $N$
is non-square-free and not a power of a prime, and $X_0^*(N)$ is of
genus $\leq 1$, or is hyperelliptic or is bielliptic. These levels
$N$ are listed in Theorem \ref{levelstudy2}. For such list we omit:
$N=420$ (a product involving four primes, see the results of \S 9 to
discard $N=420$), and $N=12,18,20,24$ because $g_N\leq 1$.

Recall the definition of $\varpi_i$ from Section \ref{2} (xii).

\subsection{$N$ a product of two primes}

\begin{small}
$$
\begin{array}{|c|c|c|c|c|c|}
 \hline N   & Prime factor& g_{X_0(N)} & g_{X_0(N)/<\varpi_{1}>} & g_{X_0(N)/<\varpi_{2}>} & g_{X_0(N)/<\varpi_{1}*\varpi_{2}>}=g_{X_0^+(N)} \\
 \hline
 \hline
  g_{X^{*}_0(N)}=0& & & & & \\
  \hline
 28  & 2^2*7    & \colorredbold{2}          & 1                       & 0                       & 1                       \\
 44  & 2^2*11    & 4          & \colorredbold{2}                       & 1                       & 1                       \\
  45  & 3^2*5    & 3          & 1                       & 1                       & 1                       \\
  50  & 2*5^2    & \colorredbold{2}          & 1                       & 1                       & 0                       \\
  54  & 2*3^3    & 4          & \colorredbold{2}                       & 1                       & 1                       \\
  56  & 2^3*7    & 5          & \colorredbold{3}                       & 1                       & 1                       \\
   92  & 2^2*23    & 10          & \colorredbold{5}                       & 1                       & \colorredbold{4}                       \\
  \hline
  \hline
  g_{X^{*}_0(N)}=1& & & & & \\
  \hline
     40  & 2^3*5    & \colorredbold{3}          & \colorredbold{2}                       & \colorredbold{2}                       & 1                       \\
    48  & 2^4*3    & \colorredbold{3}          & \colorredbold{2}                       & \colorredbold{2}                       & 1                       \\
    52  & 2^2*13    & 5          & \colorredbold{2}                       & 3                       & \colorredbold{2}                       \\
    63  & 3^2*7    & 5          & \colorredbold{3}                       & 3                       & 1                       \\
    68  & 2^2*17    & 7          & 3                       & 4                       & \colorredbold{2}                       \\
    72  & 2^3*3^2    & 5          & \colorredbold{2}                       & 3                       & \colorredbold{2}                       \\
    75  & 3*5^2    & 5          & 3                       & 3                       & 1                       \\
    76  & 2^2*19    & 8          & 4                       & 3                       & 3                       \\
    80  & 2^4*5    & 7          & 3                       & 4                       & \colorredbold{2}                       \\
    96  & 2^5*3    & 9          & 3                       & 5                       & 3                       \\
    98  & 2*7^2    & 7          & 4                       & 3                       & \colorredbold{2}                       \\
    99  & 3^2*11    & 9          &5                       & 3                       & 3                       \\
    100  & 2^2*5^2    &7          &\colorredbold{2}                       &4                       &3                       \\
    108  & 2^2*3^3    &10          & 4                       & 4                       & 4                       \\
    124  & 2^2*31    &14          & 7                       & 3                       & 6                      \\
     188  & 2^2*47    &22          & 11                       & 4                       &9                     \\
  \hline
\end{array}
$$
\end{small}

\begin{small}
$$
\begin{array}{|c|c|c|c|c|c|}
 \hline N   & Prime factor& g_{X_0(N)} & g_{X_0(N)/<\varpi_{1}>} & g_{X_0(N)/<\varpi_{2}>} & g_{X_0(N)/<\varpi_{1}*\varpi_{2}>}=g_{X_0^+(N)} \\
  \hline
  \hline
  g_{X^{*}_0(N)}=2& & & & & \\
  \hline
  \colorredbold{88}  & 2^{3}*11    & 9          & 4                       & 5                       & 4                       \\
  \colorredbold{104}  & 2^{3}*13    & 11          & 6                       & 6                       & \colorredbold{3}                       \\
  \colorredbold{112} & 2^4*7       & 11         & 6                       & 4                       & 5                       \\
  \colorredbold{116} & 2^2*29      & 13         & 6                       & 7                       & 4                       \\
  \colorredbold{117} & 3^2*13      & 11         & 5                       & 6                       & 4                       \\
  \colorredbold{135} & 3^3*5      & 13         & 7                       & 6                       & 4                       \\
  \colorredbold{147} & 3*7^2      & 11         & 5                       & 6                       & 4                       \\
  \colorredbold{153} & 3^2*17      & 15         & 7                       & 6                       &6                        \\
  \colorredbold{184} &2^3*23       &  21        & 11                      & 5                       & 9                      \\
  \colorredbold{284}  &2^2*71     &  34        & 17                      & 7                       & 14                      \\
  \hline
  \hline
  g_{X^{*}_0(N)}=3& & & & & \\
  \hline
%
%
  \colorredbold{136}  &2^3*17       &  15        & 7                      &  8                      & 6                      \\
  144  &2^4*3^2       &  13        & 7                      &7                        & 5                      \\
  152  &2^3*19       &  17        & 8                      &9                        & 6                      \\
  164  &2^2*41       &  19        & 9                      &10                        & 6                      \\
  \colorredbold{171}  &3^2*19       &  17        & 9                      &9                        & 5                      \\
  189  &3^3*7       &  19        & 8                      &10                        & 7                      \\
  196  &2^2*7^2       &  17        & 7                      &9                        & 7                      \\
  \colorredbold{207}  &3^2*23       &  21        & 11                      &8                        & 8                      \\
  236  &2^2*59       &  28        & 14                      &10                        & 10                      \\
  245  &5*7^2       &  21        & 10                      &9                        & 8                      \\
  248  &2^3*31       &  29        & 15                      &9                        & 11                      \\
  \hline
%
%
%
 \hline
 g_{X^{*}_0(N)}=4& & & & & \\
  \hline
  148  &2^2*37       &  17        & 8                      &9                        & 8                      \\
  160  &2^5*5       &  17        & 9                      &9                        & 7                      \\
  172  &2^2*43       &  20        & 10                      &9                        & 9                      \\
  \colorredbold{176}  &2^4*11       &  19        & 10                      &  10                      & 7                      \\
  200  &2^3*5^2       &  19        & 10                      &10                        & 7                      \\
  224  &2^5*7       &  25        & 13                      &11                        & 9                      \\
  225  &3^2*5^2       &  19        & 9                      &10                        & 8                      \\
   242 &2*11^2       &  22        & 11                      &10                        & 9                      \\
   275 &5^2*11       &  25        & 13                      &11                        & 9                      \\
  \hline
  \hline
  g_{X^{*}_0(N)}=5& & & & & \\
  \hline
  \colorredbold{279}  &3^2*31       &  29        & 15                      &15                        & 9                      \\
  \hline
  \end{array}
$$
\end{small}

In the previous tables $\colorredbold{N}$ means that $X_0^*(N)$ is
an hyperelliptic curve, and $\colorredbold{k}$ with $k$ an integer
means that the quotient modular curve with such genus is an
hyperelliptic curve.

\newpage
\vspace{0.4cm}

\subsection{$N$ a product of three primes}

\hspace{-1cm}

\begin{center}
\begin{tiny}
$$
\begin{array}{|l|c||c|c|c|c|c|c|c|c|}
 \hline N                            &              60&84&90&120&126&132&140&150&156\\   \hline
  Factor(N)    &     2^2*3*5&2^2*3*7&2*3^2*5&2^3*3*5&2*3^2*7&2^2*3*11&2^2*5*7&2*3*5^2&2^2*3*13\\  \hline
 W_N=\{id\}                         &            7&11&11 &17 &17&19&19&19&23 \\    \hline
  \langle \varpi_1\rangle            &             \colorredbold{3}&5& 6&9&9&9&9&10&11 \\  \hline
 \langle{\varpi_2}\rangle            &             4&5& 5&9&9&10&10&10&11\\  \hline
 \langle{\varpi_3}\rangle            &           4&6&    5&9&9&7&10&10&12\\   \hline
 \langle{\varpi_1\varpi_2}\rangle &            \colorredbold{4}&5&  6&7&9&10&8&9&11\\   \hline
 \langle{\varpi_1\varpi_3}\rangle &             \colorredbold{2}&6&  6&9&7&7&10&7&12\\   \hline
 \langle{\varpi_2\varpi_3}\rangle &             1&6&  5&5&5&10&7&7&6  \\   \hline
 \langle{\varpi_1\varpi_2\varpi_3}\rangle&  \colorredbold{3}&4&  4&7&7&8&7&8 &10 \\   \hline
 {\langle\varpi_{1},\varpi_{2}\rangle} &         \colorredbold{2}&\colorredbold{2}&  3&4&5&5&4& 5 &5\\   \hline
 {\langle\varpi_{1},\varpi_{3}\rangle} &            1&3& 3 &5&4&2&5&4&6    \\  \hline
 {\langle\varpi_{2},\varpi_{3}\rangle}&              1&3&\colorredbold{2}&3&\colorredbold{3}&4&4&4&3  \\   \hline
 {\langle\varpi_1,\varpi_2\varpi_3\rangle} &         0&\colorredbold{2}& \colorredbold{2} &\colorredbold{2}&\colorredbold{2}&4&\colorredbold{2}&3&\colorredbold{2}  \\  \hline
 {\langle\varpi_2,\varpi_1\varpi_3\rangle} &         1&\colorredbold{2}& \colorredbold{2} &4&\colorredbold{3}&3&4&3&5 \\  \hline
 {\langle\varpi_3,\varpi_1\varpi_2\rangle} &          \colorredbold{2}&\colorredbold{2}& \colorredbold{2} &\colorredbold{3}&4&3&3&4&5  \\   \hline
 {\langle\varpi_1\varpi_2,\varpi_1\varpi_3\rangle}& 0&3& 3
 &\colorredbold{2}&\colorredbold{2}&4&3&\colorredbold{2}&3 \\ \hline
 B(N)&0&1&1&1&1&1&1&1&1\\ \hline

\end{array}$$
\end{tiny}
\end{center}

\begin{center}

\begin{tiny}
$$
\begin{array}{|l|c||c|c|c|c|c|c||c|c|}
 \hline N                            &   220 &   168&         180 &198 &204 &276 & 380& 234&240\\    \hline
  Factor(N)                       & 2^2*5*11 &       2^3*3*7&     2^2*3^2*5&2*3^2*11 &2^2*3*17&2^2*3*23 & 2^2*5*19 & 2*3^2*13&2^4*3*5  \\  \hline
 W_N=\{id\}&              31 &               25&      25 & 29 & 31 & 43 & 55&  35 &37   \\    \hline
 \langle{\varpi_1}\rangle           &   15 &           13& 11 & 14 &15 & 21 & 27&  18 &19  \\  \hline
 \langle{\varpi_2}\rangle         &   16 &          13& 13  & 15 &16  & 22 & 28& 17  &19  \\  \hline
\langle{\varpi_3}\rangle         &   13 &           13&13 &  12 &16
&13& 25& 18 & 19  \\   \hline \langle{\varpi_1\varpi_2}\rangle& 16 &
11& 11 & 14&16 &22 & 28&  18 &19  \\   \hline
\langle{\varpi_1\varpi_3}\rangle &13 & 9& 11 & 15 &12 &19 & 25 & 15 &15 \\
\hline
\langle{\varpi_2\varpi_3}\rangle&10 & 13& 13 & 12  & 13 &22 & 16 & 16 & 15  \\
\hline
 \langle{\varpi_1\varpi_2\varpi_3}\rangle & 14 & 11&   11& 13  &13  & 18 & 24 & 15 &17  \\   \hline
{\langle\varpi_{1},\varpi_{2}\rangle} & 8 &  6&             5 &  7
&8 &11  & 14  & 9 &10  \\   \hline
{\langle\varpi_{1},\varpi_{3}\rangle} & 5 & 5& 5 & 6  &6 & 5& 11 & 8
&8  \\  \hline
 {\langle\varpi_{2},\varpi_{3}\rangle} & 4 &          7&    7 & 5  &7 & 7 & 7 &  8 &8  \\   \hline
{\langle\varpi_1,\varpi_2\varpi_3\rangle} & 4 &         6& 5  & 5  &5 & 9& 6 &  7  & 7 \\
\hline {\langle\varpi_2,\varpi_1\varpi_3\rangle} & 6 &    4&      5
&7 &5 & 8& 11 & 6 & 7 \\  \hline
{\langle\varpi_3,\varpi_1\varpi_2\rangle} & 6 & 5& 5 &  5 &7 & 5& 11
& 8 &9 \\   \hline {\langle\varpi_1\varpi_2,\varpi_1\varpi_3\rangle}
&4 & \colorredbold{4}& 5& 6 &5 & 10& 7 & 7&6   \\ \hline
B(N)&1 & \colorredbold{2}&\colorredbold{2}&\colorredbold{2}&\colorredbold{2}&\colorredbold{2}&\colorredbold{2}&3&3\\
  \hline
\end{array}$$
\end{tiny}
\end{center}

\begin{center}
\begin{tiny}
$$
\begin{array}{|l|c|c|c|c|c|c||c|c|c|}
 \hline
 N                            &252 & 294 & 312&               315 &348 &476 &228&260& 264 \\    \hline
 Factor(N)                & 2^2*3^2*7 & 2*3*7^2 & 2^3*3*13&           3^2*5*7&2^2*3*29&2^2*7*17& 2^2*3*19&2^2*5*13&2^3*3*11  \\  \hline
W_N=\{id\}                       & 37 & 41 & 49&             41 &
55& 67& 35& 37&41 \\    \hline \langle{\varpi_1}\rangle          &
17 & 21 & 25& 21 & 27& 33& 17& 17&19 \\  \hline
\langle{\varpi_2}\rangle & 19 & 21 & 25& 19 & 28&34& 17& 19& 21\\
\hline
\langle{\varpi_3}\rangle & 19 & 21 & 25 & 21 & 28&34 & 18&19&21\\
\hline
\langle{\varpi_1\varpi_2}\rangle & 19 & 21 & 25& 19 & 28&34 & 17&19&19\\
\hline \langle{\varpi_1\varpi_3}\rangle & 19 & 17 & 19&           21
& 22&30 & 18& 19& 21
\\   \hline
 \langle{\varpi_2\varpi_3}\rangle  &  13 & 17 & 17&           17 & 19&19 & 18& 19&21\\   \hline
 \langle{\varpi_1\varpi_2\varpi_3}\rangle & 17 & 17 &  23& 17& 25 &29 & 16& 15&17\\   \hline
 {\langle\varpi_{1},\varpi_{2}\rangle} & 9 & 10 &   13&          9 & 14& 17 &8 &9 &9\\   \hline
 {\langle\varpi_{1},\varpi_{3}\rangle} &  9 & 9 &   10&         11 &11&15  &9 &9 &10\\  \hline
 {\langle\varpi_{2},\varpi_{3}\rangle} & 7 & 9 &   9&         8 & 10& 10& 9& 10&11\\   \hline
 {\langle\varpi_1,\varpi_2\varpi_3\rangle}  & 5 & 8 &   8&    7  & 8&7  & 8&7 &8\\  \hline
 {\langle\varpi_2,\varpi_1\varpi_3\rangle}  & 9 & 7 &    9&   8  &10 &13 & 8&8 &9\\  \hline
 {\langle\varpi_3,\varpi_1\varpi_2\rangle} & 9 & 9 &    12&   8 &13& 15 & 8& 8&8\\   \hline
 {\langle\varpi_1\varpi_2,\varpi_1\varpi_3\rangle} & 7 & 7 & 6& 8&7&  8& 9 &10 &10\\ \hline
 B(N) &\colorredbold{3} & 3 &3&\colorredbold{3}&3&3&4&4&4\\
  \hline
\end{array}$$
\end{tiny}
\end{center}

\begin{tiny}
$$
\begin{array}{|l|c|c|c|c||c|c|c||c|}
 \hline
  N                           & 280 & 300 & 306 & 342 &  364       &444 &495 &558       \\    \hline
 Factor(N)           & 2^3*5*7 & 2^2*3*5^2 & 2*3^2*17 & 2*3^2*19  &    2^2*7*13    &2^2*3*37 & 3^2*5*11&2*3^2*31       \\  \hline
 W_N=\{id\}                       & 41 & 43 & 47 & 53 &  51  &71 &65 & 89          \\    \hline
 \langle{\varpi_1}\rangle          & 21 & 19 & 23 & 26 &   25   &35 &33  &45  \\  \hline
  \langle{\varpi_2}\rangle            & 21 & 22 & 23 & 27 &    26 &35 &33 &45 \\  \hline
 \langle{\varpi_3}\rangle            & 21 & 22 & 22 & 27&  26   &36 &29 &45 \\   \hline
 \langle{\varpi_1\varpi_2}\rangle & 19 & 22 & 23 & 26&    26  &35 &33 & 45\\   \hline
 \langle{\varpi_1\varpi_3}\rangle & 17 & 22 & 24 & 24  &    24  & 36&29 & 41\\   \hline
 \langle{\varpi_2\varpi_3}\rangle & 21 & 19 & 22 & 21&     23   & 24&33 & 33\\   \hline
 \langle{\varpi_1\varpi_2\varpi_3}\rangle & 19 &  19 & 20 & 24& 23 &32 &25 & 41\\   \hline
 {\langle\varpi_{1},\varpi_{2}\rangle}    & 10 & 10 & 11 & 13&  13   &17 &17 &23\\   \hline
 {\langle\varpi_{1},\varpi_{3}\rangle} & 9 & 10 & 11 & 12 &     12    &18 &13 & 21\\  \hline
 {\langle\varpi_{2},\varpi_{3}\rangle} & 11 & 10 & 10 & 11&    12    &12 &15 &17 \\   \hline
  {\langle\varpi_1,\varpi_2\varpi_3\rangle} & 10 & 7 & 9 & 9 & 10   &10 &13 &15 \\  \hline
  {\langle\varpi_2,\varpi_1\varpi_3\rangle} & 8 & 10 & 10 & 11 &  11  &16 & 11&19 \\  \hline
  {\langle\varpi_3,\varpi_1\varpi_2\rangle} & 9 & 10 & 9 & 12&   12  &16 & 11&21 \\   \hline
  {\langle\varpi_1\varpi_2,\varpi_1\varpi_3\rangle} & 8 & 10 & 11 & 9&11 & 12& 15& 15 \\ \hline
  B(N) &4&4&4&4&5&5&5&7\\
  \hline
\end{array}$$
\end{tiny}

\section{Table of bielliptic quotient curves}

We consider $(N,W_N)$ with $N$ non-square free of genus $g_W\geq 2$
which is a bielliptic curve. We assume that $W_N\neq B(N)$ and
$W_N\neq\langle w_N\rangle$, and non-trivial. Let $F$ be the field
where are defined all automorphism elements of $End(J_0(N)^{W_N})$,
which is  $\Q$ or a quadratic field $K$. Such a field is determined
from the $\Q$-isogeny decomposition of the Jacobian of $(N,W_N)$ by
using mainly Propositions \ref{defi} and \ref{Pyle}. We indicate the
elliptic factors through $K$-isogeny if the field $K$ is not $\Q$
(when $g_{W_N}\leq 5$). By $(\omega,E)$ a couple where $\omega$
indicates a bielliptic involution as an element of
$Aut(X_0(N)/W_N)\; mod\; W_N$, and $E$ denotes the corresponding
bielliptic quotient modulo $\Q$-isogeny (not $\Q$-isomorphism!) and
$\color{red}{(\omega,E)}$ a couple as before but defined all over
$K$ in particular $E$ as $K$-isogeny. We add $\cdots$ in the tables
to indicate that the modular curve $X_0(N)/W_N$ could have more
bielliptic involutions.

For all bielliptic $(N,W_N)$ that are not hyperelliptic, such that
no Atkin-Lehner appears as a bielliptic involution (or $g_{W_N}\geq
6$ or that are listed in Theorem \ref{mainpq}), we determine all
bielliptic involutions by use of computations related to Petri's
theorem, (in the situations where Petri's theorem applies). This
procedure can be done also for the rest of quotient curves where we
can apply the Petri methodology. One may consult all the
computations done in MATHEMATICA in:
\begin{verbatim}
https://github.com/FrancescBars/Mathematica-files-on-Quotient-Modular-Curves
\end{verbatim}

For $(N,W_N)$ hyperelliptic and bielliptic, with hyperelliptic
involution $u$, the bielliptic involutions can be computed by MAGMA
or SAGE by using the explicit equation given by in \cite{Hata} and
\cite{FM}. In such situations we do not make explicit the bielliptic
involutions except if one corresponds to an Atkin-Lehner involution.

\begin{center}
$$
\begin{array}{|c|c||c|}
\hline (N,W_N)&(w,E)&\Q-Jacobian decomp.\\
\hline(40,\langle w_8\rangle)&(w_5,E20a\sim X_0^*(40)),(u\circ
w_5,E40a)&E20a\times E40a\\
(40,\langle w_5\rangle)&(w_8,E20a),(V_2,E20a),(S_2,E20a),(w_8S_2w_8,E20a)&(E20a)^2\\
(44,\langle w_4\rangle)&(S_2w_2^{(22)}S_2,E11a),(S_2w_{22}^{(22)}S_2,E11a)&(E11a)^2\\
(48,\langle w_{16}\rangle)&(w_3,24a\sim X_0^*(48)),(u\circ
w_3,48a)&E24a\times E48a\\
&&E24a\sim_{\Q(\sqrt{-1})}E48a\\
(48,\langle w_3\rangle)&(w_{16},E24a),(V_2,E24a),(S_2,E24a),(S_2w_{16}S_3,E24a)&(E24a)^2\\
(52,\langle w_4\rangle)&(w_{13},E26b\sim X_0^*(52)),(u\circ
w_{13},E26a)&E26a\times E26b\\
(60,\langle w_{20}\rangle)&(w_3,E15a),(w_3\circ
u=w_4,E30a)&E15a\times
E30a\\
(60,\langle w_4,w_3\rangle)&(S_2w_{10}^{(30)}S_2,E15a),(S_2w_2^{(30)}S_2,E30a)&E15a\times E30a\\
(72,\langle w_8\rangle)&(w_9,E36a\sim X_0^*(72)),(,u\circ
w_9,E72a)&E36a\times
E72a\\
(84,\langle w_4,w_3\rangle)&(w_7,E42a\sim X_0^*(84)),(u\circ
w_7,E14a)&E14a\times E42a\\
(84,\langle w_4,w_{21}\rangle)&(w_7,E42a),(u\circ
w_7,E14a)&E14a\times E42a\\
(84,\langle w_3,w_{28}\rangle)&(w_7,E42a),(u\circ
w_7,E14a)&E14a\times E42a\\
(84,\langle w_7,w_{12}\rangle)&(w_3,E42a\sim X_0^*(84)),(u\circ
w_3,E21a)&E21a\times E42a\\
(90,\langle w_9,w_5\rangle)&(w_9,E30a\sim X_0^*(90)),(u\circ
w_9,E90b)&E30a\times E90b\\
(90,\langle w_2,w_{45}\rangle)&(w_9,E30a),(u\circ
w_9,E15a)&E15a\times E30a\\
(90,\langle w_9,w_{10}\rangle)&(w_5,E30a),(u\circ
w_5,E15a)&E15a\times E30a\\
(90,\langle w_5,w_{18}\rangle)&(w_9,E30a),(u\circ
w_9,E45a)&E30a\times E45a\\
(100,\langle w_4\rangle)&(w_{25},E50b\sim X_0^*(100)),(w_{25}\circ
u,E50a)&E50a\times E50b\\
&&E50a\sim_{\Q(\sqrt{5})} E50b\\
 (120,\langle w_8,w_{15}\rangle&(w_3,E20a\sim
X_0^*(120)),(u\circ w_3,E40a)&E20a\times E40a\\
(120,\langle w_{15},w_{40}\rangle&(w_3,E20a),(u\circ w_3,E120a)&E20a\times E120a\\
(126,\langle w_2,w_{63}\rangle&(w_9,E21a\sim
X_0^*(126)),(u\circ w_9,E14a)&E14a\times E21a\\
(126,\langle w_{14},w_{18}\rangle&(w_9,E21a),(u\circ w_9,E126a)&E21a\times E126a\\
(132,\langle w_4,w_{11}\rangle)&(w_3,E66b\sim
X_0^*(132)),(u\circ w_3,E66c)&E66b\times E66c\\
(140,\langle w_4,w_{35}\rangle)&(w_5,E70a\sim
X_0^*(140)),(u\circ w_5,E14a)&E14a\times E70a\\
(150,\langle w_6,w_{50}\rangle)&(w_2,E15a\sim
X_0^*(150)),(u\circ w_2,E150a)&E15a\times E150a\\
(156,\langle w_4,w_{39}\rangle)&(w_{13},E26b\sim
X_0^*(156)),(u\circ w_{13},E26a)&E26a\times E26b\\
\hline

\end{array}
$$
\end{center}
\begin{centerline}
{Table 1. Case $g_{W_N}=2$}
\end{centerline}

For genus 3 table below we list first the three hyperelliptic ones.
We remind that the automorphism group can be computed in MAGMA and
here we do
not make it explicit if is not an Atkin-Lehner involution.

\begin{center}
$$
\begin{array}{|c|c||c|}
\hline (N,W_N)&(w,E)&\Q-Jacobian decomp.\\
\hline (56,\langle w_8\rangle)&(V_2w_8,E14a)&(E14a)^2\times E56b\\
(60,\langle
w_4\rangle)&(S_2w_5^{(30)}S_2,E30a),(S_2w_6^{(30)}S_2,E15a),&(E15a)^2\times
E30a\\
&(S_2w_{30}^{(30)}S_2,E15a).&\\
(63,\langle w_9\rangle)&(w_7,E21a\sim
X_0^*(63))&E21a\times A_{f,63},\;\\
&\color{red}{(**,\tilde{E})(*,\tilde{E})}&A_{f,63}\sim_{\Q(\sqrt{-3})} \tilde{E}^2\\
\hline \hline (52,\langle w_7\rangle)&(w_8,E26b\sim
X_0^*(52)),...&(E26b)^2\times E52a\\
(63,\langle w_7\rangle)&(w_9,E21a),...&(E21a)^2\times E63a,\\
&&E21a\sim_{\Q(\sqrt{-3})}E63a\\
(68,\langle w_4\rangle)&(w_{17},E34a\sim X_0*(68)),...&(E17a)^2\times E34a\\
(72,\langle w_9\rangle)&(w_8,E36a\sim X_0^*(72)),...&E24a\times (E36a)^2\\
(75,\langle w_3\rangle)&(w_{25},E15a\sim X_0^*(75)),...&(E15a)^2\times E75a\\
(75,\langle w_{25}\rangle)&(w_3,E15a),...&E15a\times E75b\times E75c,\\
&&E15a\sim_{\Q(\sqrt{5})}E75b\\
(76,\langle w_{19}\rangle)&(w_4,E38b\sim X_0^*(76))....&(E38b)^2\times E76a\\
(80,\langle w_{16}\rangle)&(w_5,E20a\sim X_0^*(80)),...&E20a\times E40a\times E80a\\
&&E40a\sim_{\Q(\sqrt{-1})}E80a\\
 (84,\langle w_4,w_7\rangle)&(w_
3,E42a\sim
X_0^*(84)),...&(E21a)^2\times E42a\\
(84,\langle w_3,w_7\rangle)&(w_ 4,E42a),...&(E42a)^2\times E84b\\
(84,\langle w_{12},w_{21}\rangle)&(w_ 3,E42a),...&E14a\times E42a\times E84a\\
(90,\langle w_2,w_9\rangle)&(w_5,E30a\sim
X_0^*(90)...&E15a\times E30a\times E90a\\
(90,\langle w_2,w_5\rangle)&(w_9,E30a)...&(E30a)^2\times E45a\\
(90,\langle w_{10},w_{18}\rangle)&(w_5,E30a)...&E15a\times E30a\times E90c\\
&& E30a\sim_{\Q(\sqrt{-3})}E90c\\
 (96,\langle w_{32}\rangle)&(w_3,E24a\sim
X_0^*(96)),...&E24a\times E48a\times E96a\\
&&E24a\sim_{\Q(\sqrt{-1})}E48a\\
(98,\langle w_{49}\rangle)&(w_2,E14a\sim X_0^*(98)),...&E14a\times A_{f,98}\\
(99,\langle w_{11}\rangle)&(w_9,E99a\sim X_0^*(99)),...&E99a\times
E99b\times
E99d\\
(120,\langle w_3,w_5\rangle)&(w_8,E20a\sim X_0^*(120)),
...&(E20a)^2\times E24a\\
(120,\langle w_5,w_{24}\rangle)&(w_8,E20a),
...&(E20a)^2\times E30a\\
(124,\langle w_{31}\rangle)&(w_4,E62a\sim
X_0^*(124)),...&(E62a)^2\times
E124b\\
(126,\langle w_9,w_7\rangle)&(w_2,E21a\sim
X_0^*(126)),...&(E21a)^2\times E42a\\
(126,\langle w_9,w_{14}\rangle)&(w_2,E21a\sim
X_0^*(126)),...&(E21a)\times A_{f,63}\\
(132,\langle w_3,w_{44}\rangle)&( w_7,E66b\sim
X_0^*(132)),...&E11a\times E33a\times E66b\\
(132,\langle w_{11},w_{12}\rangle)&( w_7,E66b),...&E44a\times E66a\times E66b\\
(140,\langle w_7,w_{20}\rangle)&(w_4,E70a\sim
X_0^*(140)),...&A_{f,35}\times E70a\\
(140,\langle w_{35},w_{20}\rangle)&(w_4,E70a),...&E14a\times E70a\times E140a\\
(150,\langle w_2,w_{75}\rangle)&(w_{25},E15a\sim
X_0^*(150)),...&E15a\times E30a\times E50a\\
(150,\langle w_3,w_{50}\rangle)&(w_{25},E15a),...&(E15a)^2\times E75a\\
(156,\langle w_3,w_{13}\rangle)&(w_4,E26b\sim
X_0^*(156)),...&(E26b)^2\times E52a\\
(156,\langle w_{39},w_{12}\rangle)&(w_4,E26b),...& E26a\times
E26b\times E156b \\
\hline
\end{array}
$$
\end{center}
\begin{centerline}
{Table 2. Case $g_{W_N}=3$}
\end{centerline}
where $\tilde{E}$ is $\Q(\sqrt{-3})$-isogenous to $Y^2=+1 + 6
\sqrt{-3}X - 27 X^2 -  (26  +6 \sqrt{-3}) X^3$.


\begin{center}
$$
\begin{array}{|c|c||c|}
\hline (N,W_N)&(w,E)&\Q-Jacobian decomp.\\
\hline (60,\langle w_3\rangle)&(w_5,E20a),(w_{20},E15a),....&(E15a)^3\times E20a\\
(60,\langle w_5\rangle)&(w_4,E30a),(w_3,E20a),...&(E20a)^2\times (E30a)^2\\
(68,\langle w_{17}\rangle)&(w_4,E34a\sim
X_0^*(68)),...&(E34a)^2\times
A_{f, x^2 - 2x - 2},\dim(A_f)=2\\
(76,\langle w_4\rangle)&(w_{19},E38b\sim
X_0^*(76)),...&(E19a)^2\times
E38a\times E38b\\
(80,\langle w_5\rangle)&(w_{16},E20a\sim X_0^*(80)),...&(E20a)^3\times E80b\\
&&E20a\sim_{\Q(\sqrt{-1})}E80b\\
 (98,\langle w_2\rangle)&(w_{49},E14a\sim
X_0^*(98)),...&(E14a)^2\times
E49a\times E98a\\
&&E14a\sim_{\Q(\sqrt{-7})}E98a\\
 (100,\langle w_{25}\rangle)&(w_4,E50b\sim
X_0^*(100)),...&E20a\times(E50b)^2\times E100a\\
&&E20a\sim_{\Q(\sqrt{5})} E100a\\
 (108,\langle w_4\rangle)&(w_{27},E54b\sim
X_0^*(108)),...&(E27a)^2\times
E54a\times E54b\\
&&E54a\sim_{\Q(\sqrt{-3})}E54b\\
(108,\langle
w_{27}\rangle)&(w_{4},E54b),...&E36a\times(E54b)^2\times
E108a\\
(112,\langle w_7\rangle)&(S_2,E56a),(w_{16}S_2w_{16},E56a)&(E56a)^2\times E112a\times E112c\\
(120,\langle w_8,w_3\rangle)&(w_5,E20a\sim
X_0^*(120))...&(E15a)^2\times E20a\times E40a\\
(120,\langle w_3,w_{40}\rangle)&(w_5,E20a)...&(E15a)^2\times E20a\times E24a\\
(126,\langle w_2,w_7\rangle)&(w_9,21a\sim
X_0^*(126))...&(E21a)^2\times E63a\times E126b\\
&&E21a\sim_{\Q(\sqrt{-3})}E126b\\
(126,\langle w_7,w_{18}\rangle)&(w_9,E21a)...&(E21a)^2\times E42a\times E63a\\
&& E21a\sim_{\Q(\sqrt{-3})}E63a\\
 (132,\langle w_3,w_{11}\rangle)&(w_4,E66b\sim
X_0^*(132)),...&E44a\times(E66b)^2\times E132b\\
(132,\langle w_4,w_{33}\rangle)&(w_3,E66b),...&(E11a)^2\times E66b\times E66c\\
(132,\langle w_{12},w_{33}\rangle)&(w_4,E66b),...&E11a\times E66b\times E66c\times E132a\\
(140,\langle w_4,w_5\rangle)&(w_7,E70a\sim
X_0^*(140)),...&E14a\times(E35a)^2\times E70a\\
(140,\langle w_5,w_7\rangle)&(w_4,E70a),...&E20a\times(E70a)^2\times E140b\\
(140,\langle w_5,w_{28}\rangle)&(w_4,E70a),...&E14a\times E20a\times E35a\times E70a \\
(150,\langle w_2,w_{25}\rangle)&(w_3,E15a\sim
X_0^*(150)),...&E15a\times E30a\times E75b\times E75c\\
&&E15a\sim_{\Q(\sqrt{5})}E75b\\
(150,\langle w_3,w_{25}\rangle)&(w_2,E15a),...&(E15a)^2\times E50b\times E150c\\
(150,\langle w_{25},w_{6}\rangle)&(w_2,E15a),...&E15a\times E50b\times E75b\times E75c\\
&& E15a\sim_{\Q(\sqrt{5})} E75b\\
(168,\langle w_3,w_{56}\rangle)& (V_2w_{168},E14a),(V_2w_8,E24a)&E14a\times E24a\times E42a\times E84b\\
(188,\langle w_{47}\rangle)&(w_4,E94a\sim X_0^*(188))...&(E94a)^2\times A_{f,x^2 - x - 3}\; \dim(A_f)=2\\
(220,\langle w_5,w_{11}\rangle)&(w_4,E110b\sim
X_0^*(220)),...&E20a\times E44a\times(E110b)^2\\
(220,\langle w_4,w_{55}\rangle)&(w_{11},E110b),...&(E11a)^2\times E110c\times E110b\\
(220,\langle w_{44},w_{55}\rangle)&(w_{11},E110b),...&E11a\times E110a\times E110b\times E220b\\

\hline
\end{array}
$$
\end{center}
\begin{centerline}
{Table 3, case $g_{W_N}=4$}
\end{centerline}

\begin{small}
\begin{center}
$$
\begin{array}{|c|c||c|}
\hline(N,W_N)&(w,E)&\Q-Jacobian decomp.\\
\hline(84,\langle w_4\rangle)&(S_2w_{14}^{(42)}S_2,E21a)&(E14a)^2\times(E21a)^2\times E42a\\
(88,\langle
w_{11}\rangle)&(S_2,E44a),(w_8S_2w_8,E44a)&(E44a)^2\times E88a\times
A_{f,88,x^2-x+4
}\\
(90,\langle w_9\rangle)&(V_3w_{10},E15a)&(E15a)^2\times E30a\times E90a\times E90b\\
&&E90a\sim_{\Q(\sqrt{-3})}E90b\\
 (96,\langle w_3\rangle)&(w_{32},E24a=X_0^*(96))...&(E24a)^3\times
E32a\times
E96b\\
(99,\langle w_9\rangle)&(w_{11},E99a=X_0^*(99))...&(E11a)^2\times
E33a\times
E99a\times E99c\\
&&E99a\sim_{\Q(\sqrt{-3})}E99c\\
(117,\langle w_9\rangle)&(V_3w_{117},E39a)&E39a\times A_{f_1}\times A_{f_2}\\
&&A_{f_2}\sim_{\Q(\sqrt{-3})}\tilde{E}^2\\
(120,\langle w_{15}\rangle)&(V_2 w_{40},E24a),(w_8S_2 w_8,E20a),(S_2,E20a)&(E20a)^2\times E24a\times E40a\times E120a\\
(120,\langle
w_8,w_5\rangle)&(w_3,E20a=X_0^*(120))...&(E20a)^2\times(E30a)^2\times
E120b\\
(126,\langle w_{63}\rangle)&{\color{red}{(V_3,E14a),(V_3w_9,E14a)}}&E14a\times(E21a)^2\times E42a\times E126a\\
&&E14a\sim_{\Q(\sqrt{-3})}E126a\\
(126,\langle
w_2,w_9\rangle)&(w_7,E21a=X_0^*(126)),...&(E14a)^2\times
E21a\times A_{f,63,x^2-3}\\
&& A_f\sim_{\Q(\sqrt{-3})} \tilde{E}^2 \\
(132,\langle
w_4,w_3\rangle)&(w_{11},E33a=X_0^*(132))...&(E11a)^2\times(E33a)^2\times
E66b\\
(140,\langle w_4,w_7\rangle)&(w_5,E70a=X_0^*(140))...&A_{f,35}^2\times E70a\\
(150,\langle
w_2,w_3\rangle)&(w_{25},E15a=X_0^*(150))...&(E15a)^2\times
E50a\times E75a\times E150b\\
(156,\langle
w_4,w_3\rangle)&(w_{13},E26b=X_0^*(156))...&(E26a)\times(E26b)\times(E39a)^2\times
E78a\\
(156,\langle w_3,w_{56}\rangle)&(w_4,E26b)...&E26a\times E26b\times
E39a\times
E78a\times E156a\\
(156,\langle w_{13},w_{12}\rangle)&(w_4,E26b)...&(E26b)^2\times
A_{f}\times
E52a\\
(168,\langle w_8,w_7\rangle)&(V_2w_{168},E21a)&(E21a)^2\times E42a\times E84b\times E168a \\
(168,\langle w_7,w_{24}\rangle)&(V_2w_{168},E21a)&(E21a)^2\times E42a\times E56a\times E84b\\
(180,\langle w_4,w_9\rangle)&(S_2V_3w_{10}^{(90)}S_2,E15a)&(E15a)^2\times E30a\times E90a\times E90b\\
&&E90a\sim_{\mathbb{Q}(\sqrt{-3})}E90b\\
(184,\langle w_{23}\rangle)&(S_2,E92a),(w_8S_2w_8,E92a)&(E92a9^2\times E184b\times A_{f,184,x^2+x-4}\\
(220,\langle
w_4,w_{11}\rangle)&(w_5,E110b=X_0^*(220))...&(E55a)^2\times
E110b\times A_{f}\\
(252,\langle w_4,w_{63}\rangle)&{\color{red}{(V_3,E14a),(V_3w_7,E14a)}}&E14a\times(E21a)^2\times E42a\times E126a\\
&&E14a\sim_{\mathbb{Q}(\sqrt{-3})}E126a\\
 \hline
\end{array} $$
\end{center}
\begin{centerline}
{Table 4, Case $g_{W_N}=5$}
\end{centerline}
\end{small}

\begin{footnotesize}
\begin{center}
$$
\begin{array}{|c|c|c|c|}
\hline g_{W_N}&(N,W_N)&(w,E)&\Q-Jacobian decomp.\\
\hline 6&(104,\langle w_8\rangle)&(V_2w_{104},E26a)&(E26a)^2\times E26b\times E52a\times A_{f,104}\\
 &(156,\langle w_4,w_{13}\rangle)&(w_3,E26b=X_0^*(156))&(E26b)^2\times
A^2_{f,39}\\
&(168,\langle w_8,w_3\rangle)&(V_2w_{168},E14a)&(E14a)^2\times E42a\times E56b\times E84b\times E168b\\
 &(220,\langle w_5,w_{44}\rangle)&(w_{4},E110b=X_0^*(220))&E11a\times
E20a\times
A_{f}\times E110b\times E110c\\
&(220,\langle w_{11},w_{20}\rangle)&(w_4,E110b)&E44a\times
E55a\times
E110b\times A_{f}\times E220a\\
 \hline
7&(120,\langle w_{24}\rangle)&(V_2w_{40},E15a)&(E15a)^2\times(E20a)^2\times E30a\times E40a\times E120a\\
&(124,\langle w_4\rangle)&(w_{31},E62a=X_0^*(124)&(A_{f_1,31})^2\times E62a\times A_{f_3,62}\\
&(136,\langle w_{8}\rangle)&(V_2w_{136},E17a)&(E17a)^2\times E34a\times A_{f_3,64}\times A_{f_4,136}\\
 &(252,\langle w_{9},w_7\rangle)&(V_3w_7,E36a)&(E21a)^3\times E36a\times(E42a)^2\times E84b\\
 \hline
8&(220,\langle
w_4,w_5\rangle)&(w_{11},E110b=X_0^*(220))&(E11a)^2\times
A_f^2\times E110b\times E110c\\
\hline
9&(126,\langle w_{9}\rangle)&(V_3w_7,E14a)&(E14a)^2\times(E21a)^2\times E42a\times (A_{f,63})^2\\
&(171,\langle w_{9}\rangle)&(V_3w_{171},E19a)&(E19a)^2\times E57a\times E57b\times E57c\times A_{f,171},dim(A_f)=4\\
&(252,w_{9},\langle w_4\rangle)&(V_3w_7,E14a)&(E14a)^2\times (E21a)^2\times E42a\times (A_{f,63})^2\\
 \hline
10&(176,\langle w_{16}\rangle)&(V_3w_{176},E11a)&(E11a)^3\times E44a\times E88a\times A_{f_1,88}\times E176a\times A_{f_2,176}\\
 \hline
11&(188,\langle w_4\rangle)&(w_{47},X_0^*(188)=E94a)&A_{f_1}^2\times E94a\times A_{f_3},dim(A_{f_1})=4\\
\hline
\end{array}$$
\end{center}
\begin{centerline}
{Table 5, $g_{W_N}\geq 6$}
\end{centerline}
\end{footnotesize}

\noindent{Francesc Bars Cortina}\\ {Departament Matem\`atiques, Edif. C, Universitat Aut\`onoma de Barcelona\\ 08193 Bellaterra, Catalonia}\\
{francesc@mat.uab.cat}, {francescbars@gmail.com}
 \vspace{1cm}

\noindent{Mohamed Gamal Kamel Asraan Youssef}\\ {Department of Mathematics, Faculty of Science, Cairo University,\\ 12613 Giza-Egypt}\\
mohgamal@sci.cu.edu.eg

\vspace{1cm}

\noindent{Andreas Schweizer}\\
{Am Felsenkeller 61,}\\
{ 78713 Schramberg, Germany}\\



\begin{thebibliography}{BGGP05}

\bibitem[Acc94]{Accola}
Robert D.~M. Accola.
\newblock {\em Topics in the theory of {R}iemann surfaces}, volume 1595 of {\em
  Lecture Notes in Mathematics}.
\newblock Springer-Verlag, Berlin, 1994.

\bibitem[BGGP05]{BGGP}
M.~H. Baker, E.~{Gonz{\'a}lez-Jim{\'e}nez}, J.~Gonz{\'a}lez, and
B.~Poonen.
\newblock Finiteness results for modular curves of genus at least 2.
\newblock {\em Amer. J. Math.}, 127(6):1325--1387, 2005.

\bibitem[Bar99]{Ba1}
Francesc Bars.
\newblock Bielliptic modular curves.
\newblock {\em J. Number Theory}, 76(1):154--165, 1999.

\bibitem[Bar08]{Ba2}
Francesc Bars.
\newblock The group structure of the normalizer of $\Gamma_0(n)$ after
  Atkin-Lehner.
\newblock {\em Comm.Algebra}, 36(6):2160--2170, 2008.

\bibitem[Bar18]{BaMom}
Francesc Bars.
\newblock On quadratic points of classical modular curves.
\newblock In {\em Number theory related to modular curves---{M}omose memorial
  volume}, volume 701 of {\em Contemp. Math.}, pages 17--34. Amer. Math. Soc.,
  Providence, RI, 2018.

\bibitem[BG19]{BaGon}
F.~Bars and J.~Gonz\'{a}lez.
\newblock Bielliptic modular curves {$X_0^*(N)$} with square-free levels.
\newblock {\em Math. Comp.}, 88(320):2939--2957, 2019.

\bibitem[BG20]{BaGon2}
Francesc Bars and Josep Gonz\'{a}lez.
\newblock Bielliptic modular curves {$X^*_0(N)$}.
\newblock {\em J. Algebra}, 559:726--759, 2020.


\bibitem[BGK20]{BaGonKa}
F.~Bars, J.~Gonz\'alez, and M.~Kamel.
\newblock Bielliptic quotient modular curves with $N$ square-free.
\newblock {\em J.Number Theory.}, 216:380--402, 2020.

\bibitem[Magma]{magma}
{W. Bosma,  J.Cannon, and C.Playoust},
     {The {M}agma algebra system. {I}. {T}he user language},
      {Computational algebra and number theory (London, 1993)},
   {J. Symbolic Comput.},
  {Journal of Symbolic Computation},
    {24},
      (1997),
    {3-4},
     {235--265}. Programmes ran on Magma online computer V2.27-2.


\bibitem[Cre17]{Cre}
J.E. Cremona.
\newblock Elliptic curve data.
\newblock \url{http://johncremona.github.io/ecdata/}, 2017.

\bibitem[FK80]{FK}
H.~M.~Farkas and I.~Kra.
\newblock {\em Riemann Surfaces.}
\newblock Springer, Berlin-Heidelberg-New York,  1980

\bibitem[FH99]{FM}
M.~Furumoto and Y.~Hasegawa.
\newblock Hyperelliptic quotients of modular curves {$X_0(N)$}.
\newblock {\em Tokyo J. Math.}, 22(1):105--125, 1999.

\bibitem[GL98]{GL}
Josep Gonz\'alez and Joan-C. Lario.
\newblock Rational and elliptic parametrizations of {$\bold Q$}-curves.
\newblock {\em J. Number Theory}, 72(1):13--31, 1998.

\bibitem[HS91]{SiHa}
Joe Harris and Joe Silverman.
\newblock Bielliptic curves and symmetric products.
\newblock {\em Proc. Amer. Math. Soc.}, 112(2):347--356, 1991.


\bibitem[Has95]{Hata}
Y.~Hasegawa.
\newblock Table of quotient curves of modular curves {$X_0(N)$} with genus
  {$2$}.
\newblock {\em Proc. Japan Acad. Ser. A Math. Sci.}, 71(10):235--239, 1995.

\bibitem[Has97]{Ha97}
Y.~Hasegawa.
\newblock Hyperelliptic modular curves {$X^*_0(N)$}.
\newblock {\em Acta Arith.}, 81(4):369--385, 1997.



\bibitem[HS06]{HaSh06}
Y.~Hasegawa and M.~Shimura.
\newblock Trigonal quotients of modular curves {$X_0(N)$}.
\newblock {\em Proc.Japan Acad.}, 82 Ser A:15--17, 2006.

\bibitem[Jeo18]{Jeon}
Daeyeol Jeon.
\newblock Bielliptic modular curves {$X_0^+(N)$}.
\newblock {\em J. Number Theory}, 185:319--338, 2018.

\bibitem[JKS20]{JKS}
Daeyeol Jeon, Chang~Heon Kim, and Andreas Schweizer.
\newblock Bielliptic intermediate modular curves {}.
\newblock {\em J. Pure Appl. Algebra}, 224:272--299, 2020.

\bibitem[KMV11]{KMV}
T.~Kato, K. Magaard, and H.~V\"olklein.
\newblock Bi-elliptic Weierstrass points on curves of genus $5$.
\newblock {\em Indag.Math.(N.S.)}, 22:116--130, 2011.

\bibitem[Ogg74]{Ogg}
A.~P. Ogg.
\newblock Hyperelliptic modular curves.
\newblock {\em Bull. Soc. Math. France}, 102:449--462, 1974.

\bibitem[Pyl04]{Pyle}
Elisabeth~E. Pyle.
\newblock Abelian varieties over {$\Bbb Q$} with large endomorphism algebras
  and their simple components over {$\overline{\Bbb Q}$}.
\newblock In {\em Modular curves and abelian varieties}, volume 224 of {\em
  Progr. Math.}, pages 189--239. Birkh\"{a}user, Basel, 2004.

\bibitem[Sch01]{Sc}
Andreas Schweizer.
\newblock Bielliptic {D}rinfeld modular curves.
\newblock {\em Asian J. Math.}, 5(4):705--720, 2001.

\bibitem[Math]{mathematica} {Wolfram Research{,} Inc.},
   {Mathematica, {V}ersion 13.1},
  {https://www.wolfram.com/mathematica},
  {Champaign, IL, 2022}

\end{thebibliography}
\end{document}